\documentclass[10pt]{article}
\usepackage[nohead,left=1in,right=1in,top=0.8in,bottom=0.9in]{geometry}
\usepackage{lmodern}
\usepackage[dvipsnames]{xcolor}
\usepackage{amsmath,amssymb,amsthm,mathtools}
\usepackage{graphicx}
\usepackage{subcaption}
\usepackage{enumitem}
\usepackage[backend=biber,
  style=numeric,
  sorting=none,          
  maxnames=3,
  giveninits=true,       
  citestyle=numeric-comp
  ]{biblatex}
\usepackage{nicefrac}
\usepackage{hyperref}
\usepackage{pgfplots}
\pgfplotsset{compat=1.17}

\newtheorem{theorem}{Theorem}
 
\newtheorem{proposition}[theorem]{Proposition}
\newtheorem{corollary}[theorem]{Corollary}

\theoremstyle{definition}

\theoremstyle{remark}  
\newtheorem{remark}[theorem]{Remark}

\theoremstyle{definition}
\newtheorem{assumption}{Assumption}

\newcommand{\R}{\mathbb{R}}
\newcommand{\N}{\mathbb{N}}
\newcommand{\E}{\mathbb{E}}

\DeclareMathOperator*{\argmin}{arg\,min}

\begin{document}

\title{A remark on an error analysis for classical and learned Tikhonov regularization schemes}

\author{
Arne Behrens\thanks{\parbox[t]{\linewidth}{Center for Industrial Mathematics, University of Bremen, Bremen, Germany,\\
Emails: \texttt{\{abehrens,iskem,pmaass\}@uni-bremen.de}}}
\and
Meira Iske\footnotemark[1]
\and
Ming Jiang\thanks{\parbox[t]{\linewidth}{LMAM, School of Mathematical Sciences, Peking University, Beijing, China,\\
Email: \texttt{ming-jiang@pku.edu.cn}}}
\and
Peter Maass\footnotemark[1]
\and
Sebastian Neumayer\thanks{\parbox[t]{\linewidth}{Faculty of Mathematics, Chemnitz University of Technology, Chemnitz, Germany,\\
Email: \texttt{sebastian.neumayer@mathematik.tu-chemnitz.de}}}
}

\date{}

\maketitle

\begin{abstract}
This paper presents an error analysis of classical and learned Tikhonov regularization schemes for inverse problems.
We first demonstrate, both theoretically and numerically, that using a fixed regularization parameter across varying noise levels—which is a common miss-specification in practice—has only a mild impact on the reconstruction error.
As a special case, we then investigate scenarios where the true data resides in an unknown finite-dimensional subspace.
Here, our results lead to an empirically supported strategy for estimating the unknown dimension based on numerical experiments.
Finally, we examine the approach that motivated this study: a method where a sparsity-promoting term is learned from denoising tasks and subsequently applied to general inverse problems via a simple heuristic parameter selection.
The corresponding error analysis is initially developed using classical concepts and subsequently refined through a more detailed investigation of the discretized setting.
\end{abstract}


\section{Motivation}\label{sec:motivation}

This paper primarily addresses regularization methods for linear inverse problems in their functional analytical setting. The motivation, however, stems from data-driven concepts.
To be more precise, this work emerged from a discussion at ICSI 2024.
For many years, Alfred Louis served as a regular visitor and co-organizer of this conference series, which he initiated alongside Ming Jiang and Nathan Ida.
Besides the scientific merits, this gave him the opportunity to meet his Chinese colleagues and friends and to explore the country from Beijing to Inner Mongolia.

Our interest was sparked by a presentation on numerical and analytic results for neural networks applied for inverse problems~\cite{Goujon2022}.
The talk steered discussions on the parallels between machine learning and variational methods—specifically regarding training at a fixed noise level. Once trained, the network was employed as a regularization term within a Tikhonov functional for various operators and noise levels.
Interestingly, despite the somehow restricted training setting, applying a simple heuristic regularization parameter adjustment strategy yielded good results.

This paper analyzes variants of this setting with the aim to determine bounds for the resulting worst-case errors in the framework of optimal regularization schemes as outlined in \cite{Louis1989}.
First of all, we consider the case of a Tikhonov scheme $T_\alpha$, where the regularization parameter $\alpha=\alpha(\bar \delta)$ is chosen optimally for a fixed noise level $\bar \delta$. We then apply this scheme to data with a different noise level $\delta$.
At first glance one would expect that for $\delta \gg \bar \delta$ the ill-posedness of the problem kicks in and leads to poor reconstructions.
To this end, we analytically compare the worst-case error of this setting with the worst-case error for an optimally chosen regularization parameter $\alpha = \alpha(\delta)$.
Interestingly, the miss-specification has only a mild impact, which is also confirmed in our numerical simulations.
Closer to data-driven concepts, we also analyze a setting, where the data is restricted to an unknown but finite-dimensional linear subspace.
This is the setting where data-driven approaches usually excel.
As a side result of our error analysis, we provide a procedure that allows to determine this intrinsic subspace dimension.
Finally, we analyze the worst-case error of the learned scheme of~\cite{Goujon2022}, extending the stability analysis established in \cite{NeuAlt2024} for a simplified setting.
In particular, we embed this scheme into the classical worst-case error analysis as outlined above.

\section{Mathematical background on Tikhonov regularization}
We use the standard setting for linear inverse problems, i.e., we consider a linear bounded operator $A\colon X \rightarrow Y$ between Hilbert spaces $X$, $Y$.
Furthermore, we assume that $A$ is compact with singular value decomposition (SVD) $(v_j, u_j, \sigma_j)_{j \in \N}$.
The adjoint operator is denoted by $A^\ast \colon Y \rightarrow  X$. 
The corresponding inverse problem consists of recovering an approximation of the true solution $ x^\dagger \in X$ from noisy data $y^\delta \in Y$, where
\begin{equation}
    y^\dagger = A x^\dagger \quad \text{and} \quad \| y^\delta - y^\dagger \| \leq \delta.
\end{equation}
An approximate solution is typically  computed by a regularization scheme $T_\alpha \colon Y \rightarrow X$, and we frequently use the notation $x_\alpha^\delta \coloneqq T_\alpha y^\delta$.
Classical choices for $T_\alpha$ are the Landweber iteration, truncated SVD, conjugate gradient iterations or Tikhonov regularization.
We focus on the most classical Tikhonov regularization, which is defined by
\begin{equation}\label{eq:TikReg}
    T_\alpha y^\delta \coloneqq \argmin_x  \  \frac{1}{2}\| Ax - y^\delta \|^2 +  \frac{\alpha}{2} \|x\|^2 = \left(  A^\ast A + \alpha I \right )^{-1} A^\ast y^\delta .
\end{equation}
Throughout, we follow the notation and concepts of \cite{Louis1989}, which was the first to present a unified approach for regularization schemes in terms of so-called filter functions. These filter functions are based on the singular value decomposition of $A$, which has been  widely and successfully used for theoretical as well as numerical investigations, for the particular case of $A$ being the Radon transform see some classical results e.g. \cite{Louis1981, Quinto1983, Louis1984, Maass1987, Derevtsov2011}.

In this setting, the worst-case error is  defined as
\begin{equation}
    \mathrm{wc}(\alpha, \delta) \coloneqq \sup_{y^\delta\in Y}
    \bigl \{ \Vert T_\alpha y^\delta - x^\dagger \Vert  : Ax^\dagger = y^\dagger, \,\|y^\delta-y^\dagger\| \leq \delta\bigr\},
\end{equation}
which we will refine by a closed-form upper bound adapted to the Tikhonov regularization \eqref{eq:TikReg}.
Convergence rates for $x_\alpha^\delta \to x^\dagger$ in terms of $\delta$ require additional assumptions, typically stated as so-called source conditions on $x^\dagger$, see for example \cite{Louis1989,EnglHankeNeubauer1996, Hofmann1995,Rieder2000, Rieder2003,Beckmann2024}.
We adopt the following classical setting in the remainder of the paper.
\begin{assumption}\label{ass:1} There exist $z\in Y$ such that $x^\dagger = A^\ast z$ with $\|z\| \leq \rho$ for some $\rho > 0$.
\end{assumption}
The following result for Tikhonov regularization can be found (embedded in a more general framework) in the proof of~\cite[Thm.~4.2.3]{Louis1989}.
We include the proof for completeness and assume for simplicity that $A$ is normalized so that its operator norm is $\|A\| = 1$.
 
\begin{theorem}\label{thm:tikh_error}
Let $X,Y$ be Hilbert spaces and $A \colon X \rightarrow Y$ be compact with $\|A\| = 1$.
Further, let $y^\dagger = Ax^\dagger$ for $x^\dagger \in X$ and assume that Assumption~\ref{ass:1} holds.
If $\alpha>0$ and $y^\delta \in Y$ with $\| y^\delta - y^\dagger\| \leq \delta$,
then the reconstruction error satisfies
\begin{equation}\label{eq:ErrorBound}
    \| T_{\alpha}y^\delta - x^\dagger \| \leq \mathrm{wc}(\alpha,\delta) \leq \begin{cases}
        \frac{1}{2}\left( \frac{\delta}{\sqrt{\alpha}} + \sqrt{\alpha} \rho \right) &\text{if} \ 0 < \alpha \leq 1, \\
        \frac{1}{1+\alpha}\left( \delta + \alpha \rho \right) &\text{if} \ \alpha > 1.
    \end{cases}
\end{equation}
\end{theorem}

\begin{proof}
We decompose the total error into data and approximation error
\begin{equation}\label{eq:ErrorDecompose}
    \| T_{\alpha} y^\delta - x^\dagger \| \leq \| T_{\alpha}(y^\delta - y^\dagger) \| + \| T_{\alpha}(A A^\ast z) - x^\dagger \|.
\end{equation}
Using the Tikhonov filter (following to the notion of A.\ Louis)
\begin{equation}
   F_{\alpha}(\sigma) = \frac{\sigma^2}{\sigma^2 +\alpha},
\end{equation}
and the SVD, we can write $T_{\alpha} = \sum_n F_{\alpha}(\sigma_j) \sigma_j^{-1} \langle \cdot, u_j \rangle v_j $.
Since $\|A\| = 1$ implies $0\le \sigma_j \le 1$, we can use the first-order optimality conditions to obtain the standard estimates
\begin{equation}\label{eq:EstFilter}
\begin{aligned}
    \sup_{\sigma_j} \left| F_{\alpha}(\sigma_j)\sigma_j^{-1} \right|
    &\leq
    \begin{cases}
        \frac{1}{2 \sqrt{\alpha}} & \text{if } 0 \leq \alpha \leq 1, \\
        \frac{1}{1 + \alpha} & \text{if } \alpha > 1,
    \end{cases}
    \\
    \sup_{\sigma_j} \left| (1 - F_{\alpha}(\sigma_j)) \sigma_j \right|
    &\leq
    \begin{cases}
        \frac{1}{2} \sqrt{\alpha} & \text{if } 0 \leq \alpha \leq 1, \\
        \frac{\alpha}{1 + \alpha} & \text{if } \alpha > 1.
    \end{cases}
\end{aligned}
\end{equation}
These bounds do not involve any knowledge on the actual spectrum of $A$.
In analogy to the proof of~\parencite[Thm.~3.4.3]{Louis1989}, we thus obtain the bounds
\begin{equation}\label{eq:EstDataError}
    \| T_{\alpha}(y^\delta - y^\dagger) \| \leq \sup_{\sigma_j} \bigl| F_{\alpha}(\sigma_j)\sigma_j^{-1} \bigr| \delta \leq \begin{cases}
        \frac{\delta}{2 \sqrt{\alpha}} &\text{if} \ 0 \leq \alpha \leq 1, \\
        \frac{\delta}{1 +\alpha} &\text{if} \ \alpha > 1,
    \end{cases}
\end{equation}
and
\begin{align}\label{eq:EstApproxError}
    \| T_{\alpha}(y^\dagger) - x^\dagger \| & = \| T_{\alpha}(A A^* z) - A^* z \| \leq  \sup_{\sigma_j} \left| (1 - F_{\alpha}(\sigma_j)) \sigma_j \right|  \rho 
    \leq \begin{cases}
        \frac{\sqrt{\alpha}}{2} \rho &\text{if} \ 0 \leq \alpha \leq 1, \\
        \frac{\alpha\rho}{1 +\alpha} &\text{if} \ \alpha > 1.
    \end{cases}
\end{align}
Inserting~\eqref{eq:EstDataError} and~\eqref{eq:EstApproxError} into \eqref{eq:ErrorDecompose} yields the result.
\end{proof}
Theorem~\ref{thm:tikh_error} gives an upper bound for the worst-case error $\mathrm{wc}(\alpha,\delta)$, which is sharp only if the upper bounds in \eqref{eq:EstFilter} are actually attained  by a singular value in the spectrum of $A$.
For clarity of presentation, we assume that this is the case, i.e., that \eqref{eq:ErrorBound} turns into
\begin{equation}\label{eq:DefWC}
    \mathrm{wc}(\alpha,\delta)  = \begin{cases}
        \frac{1}{2}\left( \frac{\delta}{\sqrt{\alpha}} + \sqrt{\alpha} \rho \right) &\text{if} \ 0 < \alpha \leq 1, \\
        \frac{1}{1+\alpha}\left( \delta + \alpha \rho \right) &\text{if} \ \alpha > 1.
    \end{cases}
\end{equation}
Now, it seems natural to choose $\alpha$ such that \eqref{eq:DefWC} becomes minimal. 
For our specific setting, this yields the choice $\alpha(\delta) \sim \delta$, as specified in the following remark, see also  \cite[Thm.~4.2.3]{Louis1989}. 

\begin{remark}\label{rem:wc_minimizer}
    For $\delta \leq \rho $, we have $\argmin_{\alpha>0} \mathrm{wc}(\alpha,\delta) = \nicefrac{\delta}{\rho}$.
    Otherwise, we have that $\alpha(\delta) = \infty$, and thus  $T_{\alpha}y^\delta =0$ is optimal.
    Hence, the parameter choice rule $\alpha \colon (0,\infty) \to (0,\infty]$ with
    \begin{equation}\label{eq:OptRegPara}
        \alpha(\delta) = \begin{cases}
             \frac{\delta}{\rho} &\text{if } \delta \leq \rho, \\
            \infty &\text{otherwise}
        \end{cases}
    \end{equation}
    is asymptotically optimal for $\delta \to 0$.
\end{remark}
Now, our main aim is to analyze the difference between the worst-case error $\mathrm{wc} (\alpha(\delta), \delta)$ for the optimally chosen regularization parameter and the error $\mathrm{wc} ( \alpha(\bar \delta), \delta)$ with a suboptimal regularization parameter $\alpha(\bar{\delta}) = \nicefrac{\bar\delta}{\rho}$.
In the latter case, we adapt the Tikhonov regularization scheme to data with noise level \smash{$\bar \delta$} and instead apply it to data with noise level \smash{$\delta \neq \bar \delta$}. 

\section{Worst-case error analysis for Tikhonov regularization}\label{sec:wc_error}

In this section, we consider the Tikhonov regularization $T_{\alpha(\bar\delta)}$, where $\alpha(\bar\delta) = \nicefrac{\bar\delta}{\rho}$ has been chosen according to \eqref{eq:OptRegPara} for the noise level \smash{$\bar \delta$}.
Then, we apply $T_{\alpha(\bar\delta)}$ to data \smash{$y^\delta$} for any noise level \smash{$\delta \neq \bar\delta$} and compare its approximation properties with the optimal reconstruction \smash{$x_{  
\alpha(\delta)}^\delta =T_{ \alpha(\delta)} y^\delta$}.
Here, we are interested in estimating the relative loss of accuracy $\mathrm{wc}(\alpha(\bar\delta), \delta)/\mathrm{wc}(\alpha(\delta), \delta)$ due to the suboptimal regularization parameter $\alpha(\bar\delta)$. 
To avoid case distinctions, we only report the results for $0< \alpha(\bar\delta), \alpha(\delta) <1$, namely the first case of \eqref{eq:DefWC}.

\begin{corollary}\label{cor:wc_error}
   For the parameter choice rule $\alpha$ in \eqref{eq:OptRegPara} and under the same assumptions as in Theorem~\ref{thm:tikh_error} as well as $0< \alpha, \alpha(\bar\delta), \alpha(\delta) <1$, the \emph{worst-case error} 
    \begin{equation}\label{eq:WC_Source}
        \mathrm{wc}(\alpha,\delta)=\frac{1}{2}\biggl( \frac{\delta}{\sqrt{\alpha}} + \sqrt{\alpha} \rho \biggr)
    \end{equation}
    satisfies  
    \begin{equation}\label{eq:rel_wcerror}
        \frac{\mathrm{wc}(\alpha(\bar\delta), \delta)}{\mathrm{wc}(\alpha(\delta), \delta)} = \frac{1}{2}\biggl(\frac{\sqrt{\delta}}{\sqrt{\bar \delta}} + \frac{\sqrt{\bar \delta}}{\sqrt{\delta}}\biggr).
    \end{equation}
    For $\delta = \lambda \bar \delta$ with $\lambda > 1$ such that $\delta < \rho$, we have 
    \begin{equation}
         \mathrm{wc}(\alpha(\bar\delta), \delta) =
             \mathcal{O}(\lambda), \quad \frac{\mathrm{wc}(\alpha(\bar\delta), \delta)}{\mathrm{wc}(\alpha(\delta), \delta)} = \mathcal{O}\bigl(\sqrt{\lambda}\bigr),
    \end{equation}
    and for $\lambda < 1$ it holds
    \begin{equation}
         \mathrm{wc}( \alpha(\bar\delta), \delta) = \mathcal{O}(1), \quad \frac{\mathrm{wc}(\alpha(\bar\delta), \delta)}{\mathrm{wc}(\alpha(\delta), \delta)} = \mathcal{O}\Big(\frac{1}{\sqrt{\lambda}}\Big).
    \end{equation}
\end{corollary}

\begin{proof}
    By the definition of the worst-case error~\eqref{eq:WC_Source} and the parameter choice rule~\eqref{eq:OptRegPara}, it holds
    \begin{equation}
        \frac{\mathrm{wc}(\alpha(\Bar \delta), \delta)}{\mathrm{wc}(\alpha(\delta), \delta)} = \frac{\sqrt{\rho}\Bigl(\frac{\delta}{\sqrt{\Bar{\delta}}} + \sqrt{\Bar \delta}\Bigr)}{2\sqrt{\rho \delta}} = \frac{\frac{\delta}{\sqrt{\Bar{\delta}}} + \sqrt{\Bar{\delta}}}{2 \sqrt{\delta}} = \frac{1}{2}\biggl( \frac{\sqrt{\delta}}{\sqrt{\bar\delta}} + \frac{\sqrt{\bar\delta}}{\sqrt{\delta}}\biggr).
    \end{equation}
    Setting $\delta = \lambda \Bar{\delta} < \rho$ gives
    \begin{equation}
        \mathrm{wc}(\alpha(\bar\delta),\delta) = \frac{\sqrt{\rho}}{2}\Bigl(\lambda \sqrt{\Bar{\delta}} +  \sqrt{\Bar{\delta}}\Bigr) = \frac{\sqrt{\rho \bar\delta}}{2}(\lambda + 1).
    \end{equation}
    Furthermore, we have 
    \begin{equation}
        \frac{\mathrm{wc}(\alpha(\bar\delta), \delta)}{\mathrm{wc}(\alpha(\delta), \delta)} = \frac{1}{2}\left( \frac{\lambda \bar \delta + \bar \delta }{\sqrt{\lambda} \bar \delta}\right) = \frac{1}{2}\left( \sqrt{\lambda} + \frac{1}{\sqrt{\lambda}}\right),
    \end{equation}
    yielding $\mathcal{O}(\sqrt{\lambda})$ as $\lambda \to \infty$ and $\mathcal{O}(1/\sqrt{\lambda})$ as $\lambda \to 0$.
\end{proof}
According to Corollary~\ref{cor:wc_error}, if we determine $\alpha$ using a reference noise level $\bar\delta$ and let the actual noise level grow as $\delta = \lambda \bar\delta$ with $\lambda \to \infty$ (until $\alpha(\delta)\geq  1$), then the relative worst-case error with respect to the optimal $\alpha$ grows merely as $\mathcal{O}(\sqrt{\lambda})$.
Thus, the deterioration caused by a severely miss-specified $\alpha$ is less dramatic than one might expect at first glance. 

\subsection{Experimental comparison setup}\label{subsec:num_exp_full}

To complement our theoretical analysis, we now turn to numerical experiments that compare classical Tikhonov regularization and data-driven reconstruction methods for linear inverse problems.
Specifically, we investigate reconstruction errors for
\begin{itemize}
\item Tikhonov regularization,
\item Learned Primal-Dual (LPD) reconstruction~\cite{AdlerOektem2018LPD}, which takes $x^{(0)}=0$ and $p^{(0)}=0$ and unrolls $K$ primal-dual iterations
\begin{equation}
\begin{aligned}
    p^{(\ell+1)} &= p^{(\ell)} + \sigma_\ell \Gamma_\ell\bigl(p^{(\ell)}, Ax^{(\ell)} - y^{\bar\delta}\bigr), \\
    x^{(\ell+1)} &= x^{(\ell)} + \tau_\ell \Lambda_\ell \bigl(x^{(\ell)}, A^\ast p^{(\ell+1)}\bigr),
\end{aligned}
\qquad \ell=0,\dots,K-1,
\end{equation}
where the dual and primal networks $\Gamma_\ell$ and $\Lambda_\ell$ and the step sizes $\sigma_\ell, \tau_\ell \in \R$ are learned.
\item Learned iterative Shrinkage Algorithm~(LISTA)~\cite{Gregor2010}, which takes $x^{(0)} = 0$ and unrolls $K$ iterations of ISTA
\begin{equation}
    x^{(\ell)} = S_\lambda (Wx^{(\ell-1)} + B y^{\bar\delta}), \qquad \ell = 1,\ldots , K,
\end{equation}
with the soft shrinkage $S_\lambda = \mathrm{ReLU}(x-\lambda)- \mathrm{ReLU}(-x-\lambda)$ and shared weights\footnote{Numerical evidence showed similar results for the case of using unshared weights across layers.} $W \in \R^{n \times n}$, $B \in \R^{n\times m}$.
Here, $\Theta = (W,B,\lambda)$ are the learnable parameters of the network.
\end{itemize}
To compare these, we consider two ill-posed inverse problems.
The first one is based on the Radon operator, while the second involves integration.
Throughout, we consider the finite-dimensional spaces $X = \R^n$ and $Y = \R^m$ and equip them with the discretized $L^2$-norms $\|\cdot\|_X^2 \coloneqq \frac{1}{n}\|\cdot\|_2^2$ and $\|\cdot\|_Y^2 \coloneqq \frac{1}{m}\|\cdot\|_2^2$.
We train both  LISTA and LPD following the details from the respective papers for multiple but fixed noise levels \smash{$\bar\delta$}, which then implicitly act as regularization parameter. By now, there also exists a rich mathematical theory for such learned methods, see e.g. \cite{Arridge2019,Dittmer2018, Arndt2023}.

\paragraph{Radon operator.}
Using MNIST images together with a discretized Radon operator is a common benchmark in the data-driven inverse problems literature \cite{Adler2022, Arndt2023, Aspri2018}.
The MNIST data set consists of $L_\mathrm{train} = 60\,000$ training images and $L_\mathrm{test} = 10\,000$ test images with size $28 \times 28 = 784$ pixels.
Following the standard pixel-based discretization of the Radon transform, see for example~\cite[Sec.~6.4]{Buzug08}, we obtain a linear operator $A_R\colon \mathbb{R}^{784} \to\mathbb{R}^{30 \cdot 41}$ that models line integrals of the images $x$ along 30 equidistant projection angles with 41 detector offsets.
Numerically, we evaluate the integrals using the \texttt{radon}-routine from the Python library \texttt{scikit-image}~\cite{vanderwalt2014} and store the resulting linear operator as a dense matrix (which is of course only possible due to the small size of the MNIST images).
The latter is normalized by its spectral norm. 
The data tuples \smash{$(x_i, y^{\bar\delta}_i)$} for the inverse problem are simulated as 
\smash{$y_i^{\bar\delta} = A_R x_i + \eta_i$} with noise $\eta_i \sim \mathcal{N}(0, \bar \delta^2 I_m)$ so that $\mathbb E (\| y_i - y^{\bar\delta}_i \|^2_Y)  = \E (\frac{1}{m}\sum_{j=1}^m |y_{i,j} - y_{i,j}^{\bar\delta}|^2 ) = \bar\delta^2$.
We evaluate all  methods on a subset of the test set containing $L = 50$ images.
For LISTA, we use $K=20$ layers.
For LPD, we perform $K=20$ iterations where both the primal and dual networks are implemented as CNNs consisting of three layers with kernel size ${3\times3}$ and $2 \rightarrow 32 \rightarrow 32 \rightarrow 1$ channels, respectively.
As activation function, we use the Parametric Rectified Linear Units~(PReLU) $\mathrm{PReLU}_{c}(x) = \mathrm{ReLU}(x) -c \mathrm{ReLU}(-x)$ with a learnable parameter $c \in \R$ that is shared across channels.
Note that LPD explicitly exploits the underlying 2D spatial structure of the problem.

\paragraph{Integration operator.}
For $n \in \N$, we define the discrete integration operator $\tilde A_I \in \mathbb{R}^{n \times n}$ via
\begin{equation}\label{eq:IntOp}
   \tilde A_I = \begin{pmatrix}
    1 & 0 & 0 & \cdots & 0 \\
    1 & 1 & 0 & \cdots & 0 \\
    1 & 1 & 1 & \cdots & 0 \\
    \vdots & \vdots & \vdots & \ddots & \vdots \\
    1 & 1 & 1 & \cdots & 1
    \end{pmatrix}_{n \times n}  
\end{equation}
and its normalized counterpart $A_I = \tilde A_I / \|\tilde A_I\|_2$.
Note that the singular values of \eqref{eq:IntOp} scale roughly as $n^{-1}$.
We set $n=50$ in~\eqref{eq:IntOp} and generate $L_\mathrm{train} = 60\, 000$ training pairs \smash{$(x_i, y_i^{\delta})_{i = 1}^{L_\mathrm{train}}$} by sampling
\begin{equation}
    z_i = \sum_{j=1}^n d_j u_j \quad \text{with} \quad d_j \sim \mathcal{U}[-1,1],
\end{equation}
where $u_j$ are the left singular vectors of $A$, in order to obtain data 
\begin{equation}\label{eq:zdata_gen}
    x_i = A^\ast z_i \quad \text{with} \quad \|z_i\|_Y = \biggl(\frac{1}{n}\sum_{j=1}^n |d_j|^2\biggr)^{1/2} = \rho_i \quad\text{for} \ i = 1,\dots,L_\mathrm{train},
\end{equation}
which fulfill the source condition in Assumption~\ref{ass:1}. We generate the associated $y_i, y_i^{\bar\delta}$ analogously to $A_R$. 
As test data, we generate $L = 50$ test tuples with the same protocol.
The LISTA approach uses $K=30$ iterations.
For the LPD, we use $K=30$ iterations and fully connected networks~(FCN) as the primal and the dual networks since we are no longer concerned with image data.
Each FCN consists of two layers of width $2n \rightarrow 128 \to 128 \rightarrow n$ and a ReLU-activation after the first two layers.

\subsection{Numerical results}\label{sec:WorstCaseNumeric}
For all three approaches, we denote by \smash{$x_{\bar\delta}^{\delta}$} the reconstruction from a measurement $y^\delta$ with noise level $\delta$ using the instance that is optimal for \smash{$\bar\delta$}.
For LISTA and LPD, this corresponds to the network that has been trained with $\bar\delta \in \{0.001, 0.01, 0.1, 0.2, 0.5, 1\}$, and for Tikhonov regularization this corresponds to choosing $\alpha(\bar\delta)=\bar\delta / \rho$, where the source constant $\rho$ is estimated from the data.
Specifically, for each image $x_i$, we derive the source constant $\rho_i$ in Assumption~\ref{ass:1} by finding the minimum-norm solution $z_i$ to $x_i = A^* z_i$ (and checking that the residual is zero).
Using the Moore-Penrose pseudoinverse, we get $\rho_i \coloneqq \|(A^*)^\dagger x_i\|_Y$.
To obtain a single, representative constant, we compute the average \smash{$\rho \coloneqq \frac{1}{L} \sum_{i=1}^L \rho_i$}.
While this characterizes the typical source condition of the samples, a strict source constant for all $x_i$ would instead require taking the more conservative—and perhaps overly pessimistic—maximum of all $\rho_i$.

Figure~\ref{fig:example} shows representative Tikhonov-based reconstructions for both operators, illustrating the experimental setup.
Across all experiments, the resulting reconstruction errors are averaged over 100 independent noise realizations per data sample.
\begin{figure}[tbp]
    \centering
    \begin{subfigure}[t]{0.49\textwidth}
    \includegraphics[width=\linewidth]{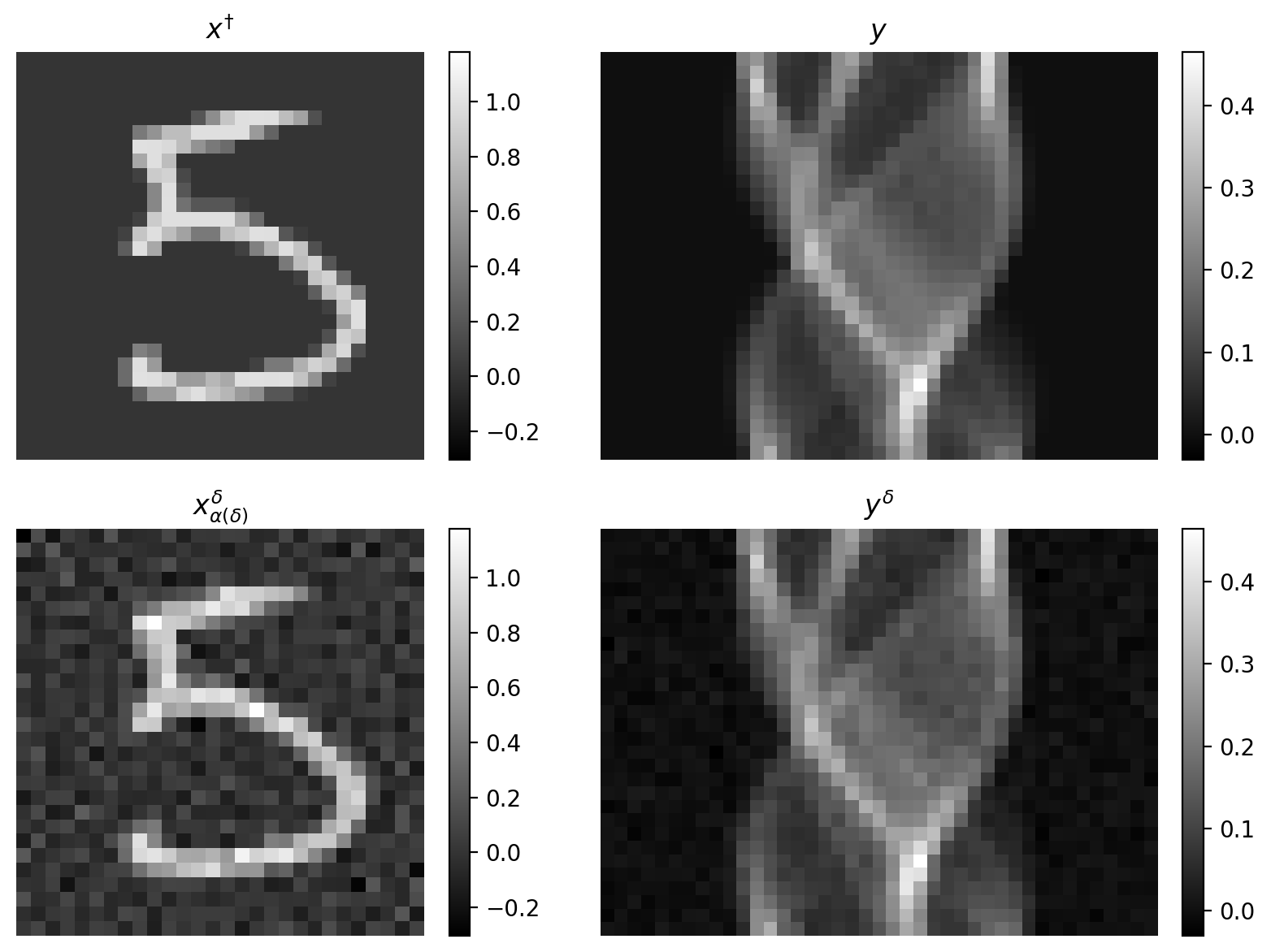}
    \end{subfigure}
    \hfill
    \begin{subfigure}[t]{0.45\textwidth}
        \includegraphics[width=0.99\linewidth]{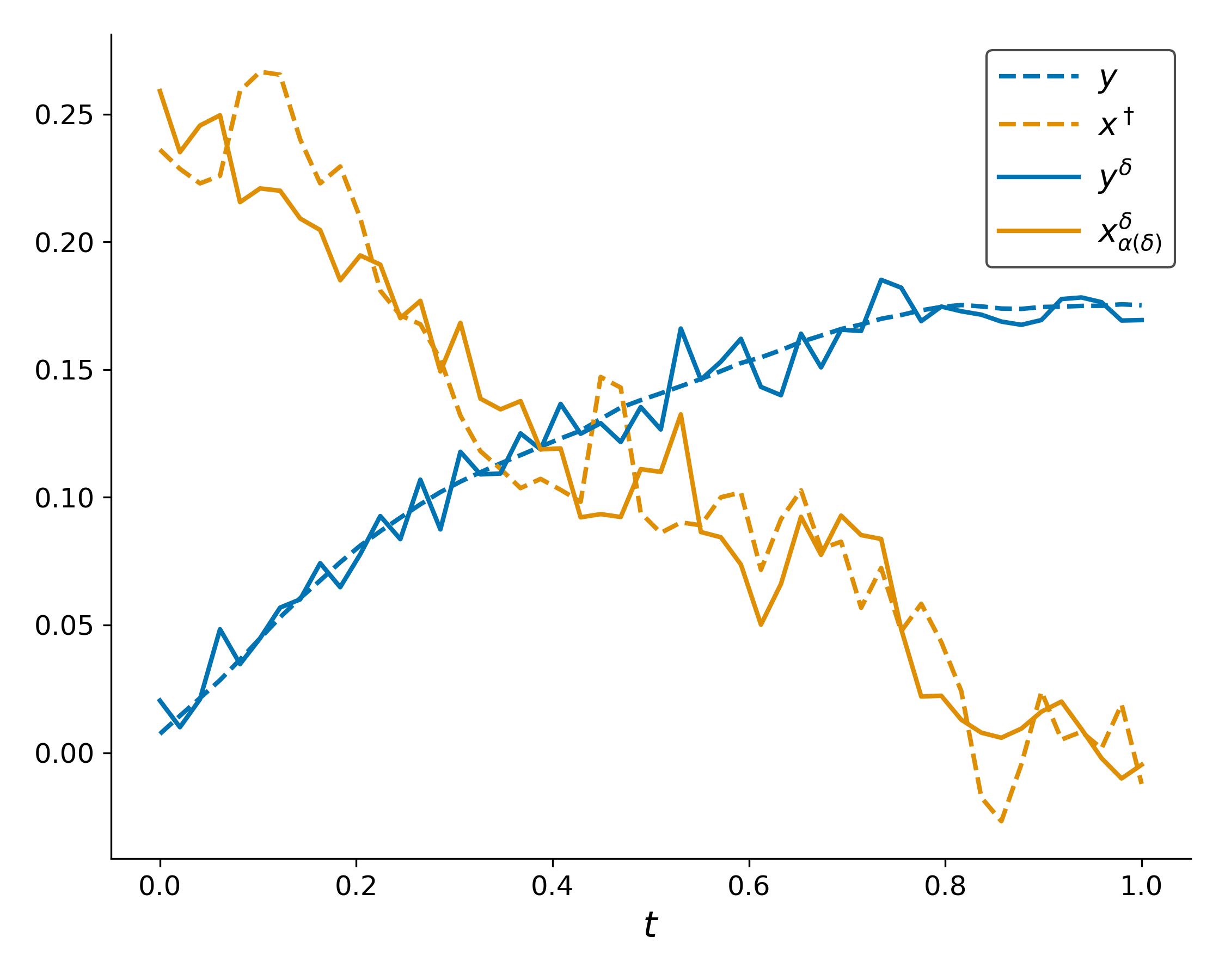}
    \end{subfigure}
    \caption{Test data $x^\dagger$ for both inverse problems together with noise-free and noise-perturbed measurements $(y,y^\delta)$ and the corresponding Tikhonov reconstruction \smash{$x^\delta_{\alpha(\delta)}$} of $y^\delta$, using $\delta = 0.01$ and $\alpha(\delta) = \nicefrac{\delta}{\rho}$. \emph{(left)}~MNIST test image and measurement generated via $A_R$,
    \emph{(right)}~test sample from~\eqref{eq:zdata_gen} and measurements generated via $A_I$ at $n=50$.
    }
    \label{fig:example}
\end{figure}
For Tikhonov regularization, the analytical worst-case errors as in Corollary~\ref{cor:wc_error} are exemplified in Figure~\ref{fig:wc_err_radon} for $A_R$.
Further, the average and relative average errors
\begin{equation}\label{eq:num_error}
    \mathcal{E}(\bar\delta,\delta) = \frac{1}{L}\sum_{i=1}^{L}\| x^{\delta}_{i,\bar\delta} - x_i\|_X  \quad \text{ and } \quad \frac{\mathcal{E}(\bar\delta,\delta)}{\mathcal{E}(\delta,\delta)}
\end{equation}
for the two inverse problems and all three reconstruction approaches are visualized in Figure~\ref{fig:wc_grid_int}.
For the LPD and LISTA approaches, one point of $\mathcal{E}(\bar\delta,\delta)$ plotted against $\bar{\delta}$ corresponds to a network trained on data \smash{$y^{\bar\delta}$} and tested on data $y^\delta$.
We compare the errors both for varying $\delta$ and $\bar \delta$.
First, we observe that the empirical Tikhonov errors always lie below the theoretical worst-case bound from Theorem~\ref{thm:tikh_error} as expected.
Although not guaranteed theoretically, the learned approaches also stay below this reference error for all choices of $\bar\delta$.

For Tikhonov regularization, plotting the error $\mathcal{E}(\bar\delta,\delta)$ versus $\delta$ at fixed $\bar\delta$ (first rows) shows a linear growth in $\delta$ for large $\delta$ across all reconstruction approaches, consistent with the data-error dominating for fixed parameter choice $\alpha$.
The varying error magnitudes are caused by the parameter $\rho$, which determines the slope of the error bounds when plotted against $\delta$.
The respective values for $\rho$ are given in the figure captions.
The intercept of $\mathcal{E}(\bar\delta,\delta)$ at $\delta=0$ quantifies the approximation error.
For the theoretical bound and the learned methods, it decreases as $\bar\delta$ decreases, while the slope with respect to $\delta$ increases with decreasing $\bar\delta$ as predicted by the decomposition in Theorem~\ref{thm:tikh_error}.

The absolute reconstruction errors of the learned methods tend to outperform the Tikhonov reconstruction.
In particular, LPD consistently yields the smallest absolute errors. 
This behavior can be attributed to the structure of the data.
MNIST images are well approximated by low-dimensional manifolds, while in the case of the integral operator the data is restricted to satisfy a source condition.
Learned methods capture such (potentially nonlinear) structures during training, whereas Tikhonov regularization does not utilize any training data. 
Unsurprisingly, the learned relative error curves reach their minima at $\delta = \bar\delta$ with a similar behavior as the worst-case Tikhonov error.
This indicates that we also have a sublinear convergence as in~\eqref{eq:rel_wcerror} for the learned methods.
Overall, the learned methods seem to admit a very similar behavior as Tikhonov regularization with respect to noise mismatch.

\begin{figure}[tbp]
    \centering
    \begin{subfigure}[t]{0.32\textwidth}
    \includegraphics[width=\linewidth]{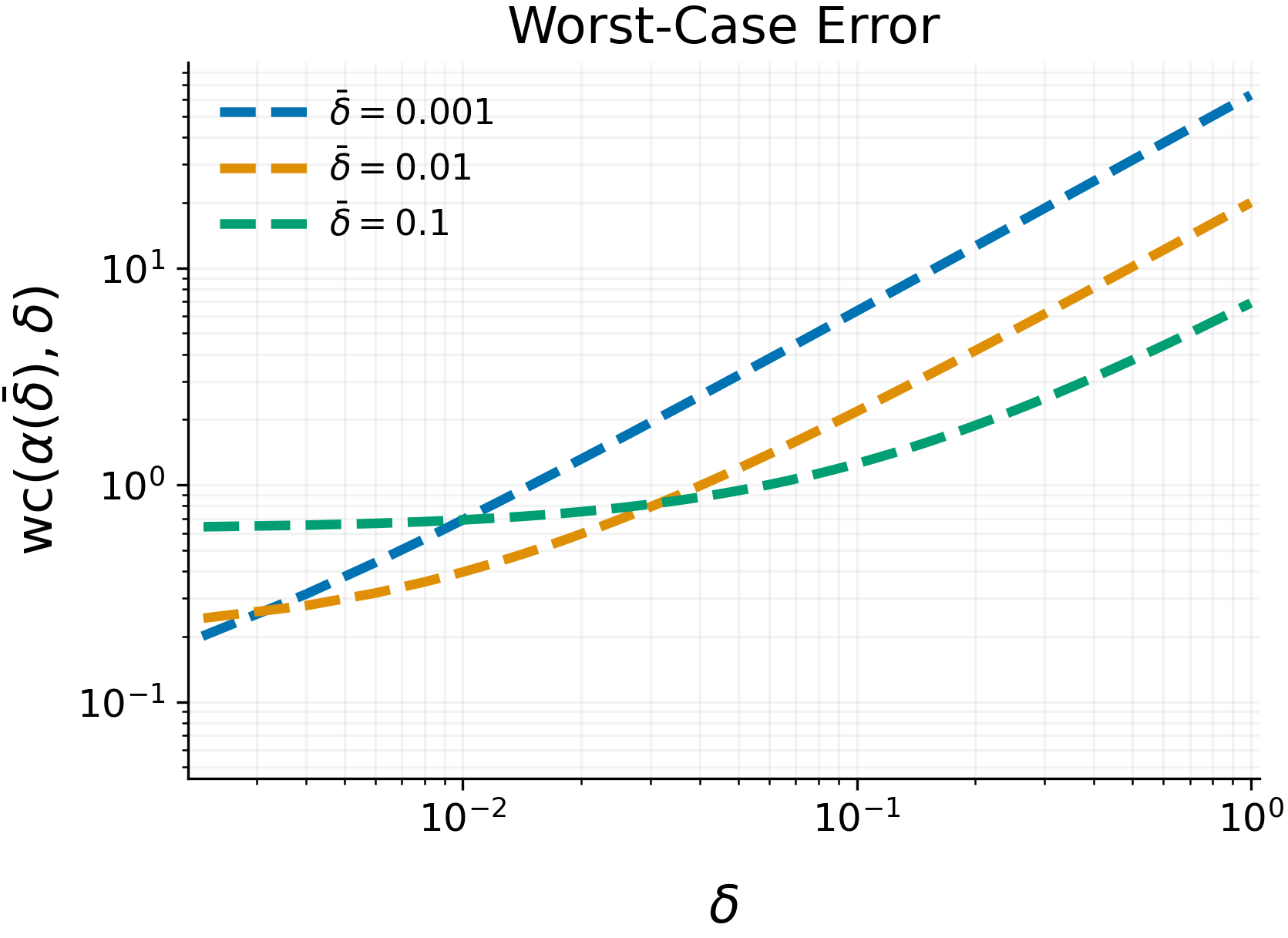}
    \end{subfigure} 
    \hfill
    \begin{subfigure}[t]{0.32\textwidth}
     \includegraphics[width=\linewidth]{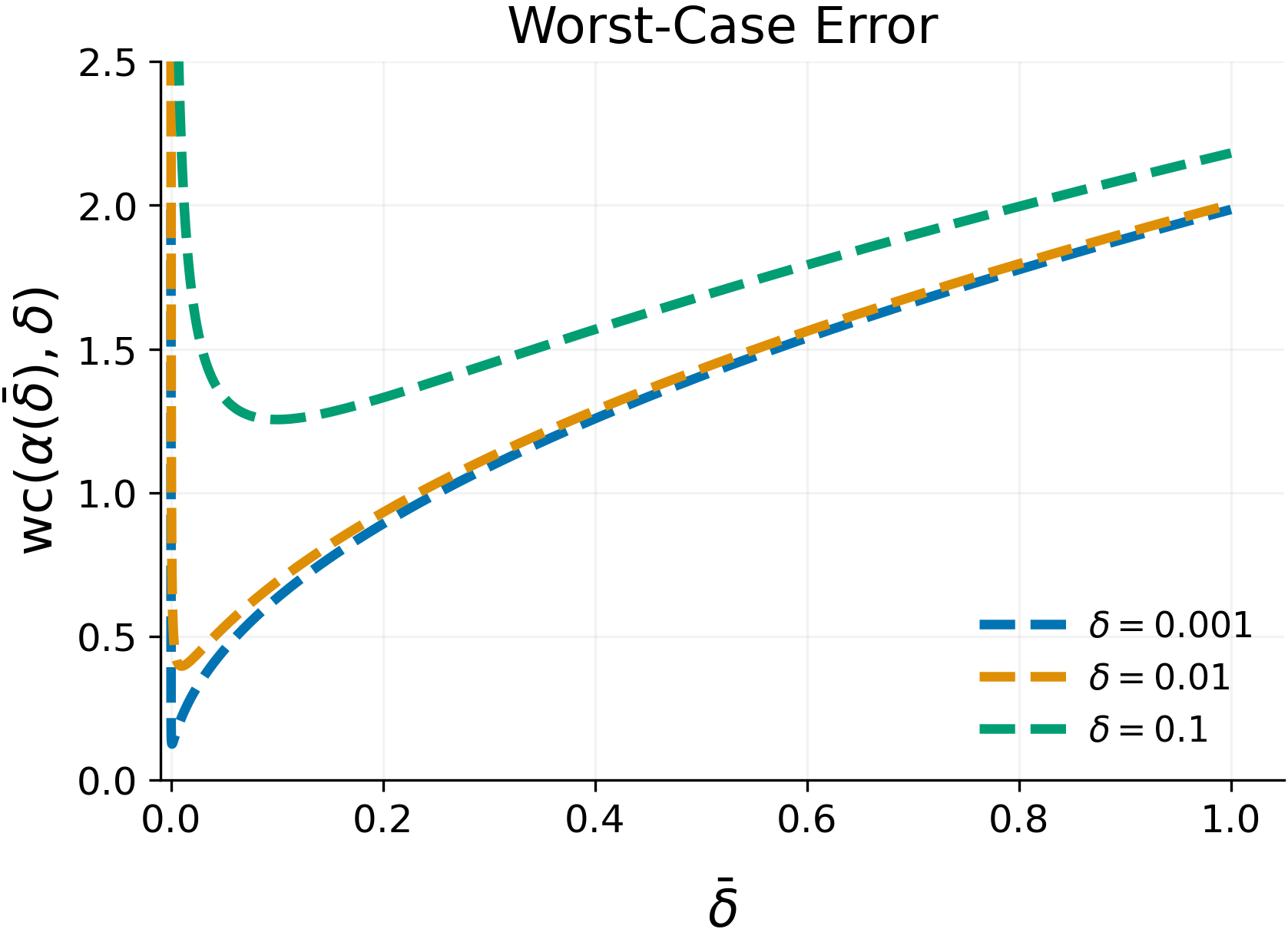}
    \end{subfigure}
    \hfill
    \begin{subfigure}[t]{0.32\textwidth}
     \includegraphics[width=\linewidth]{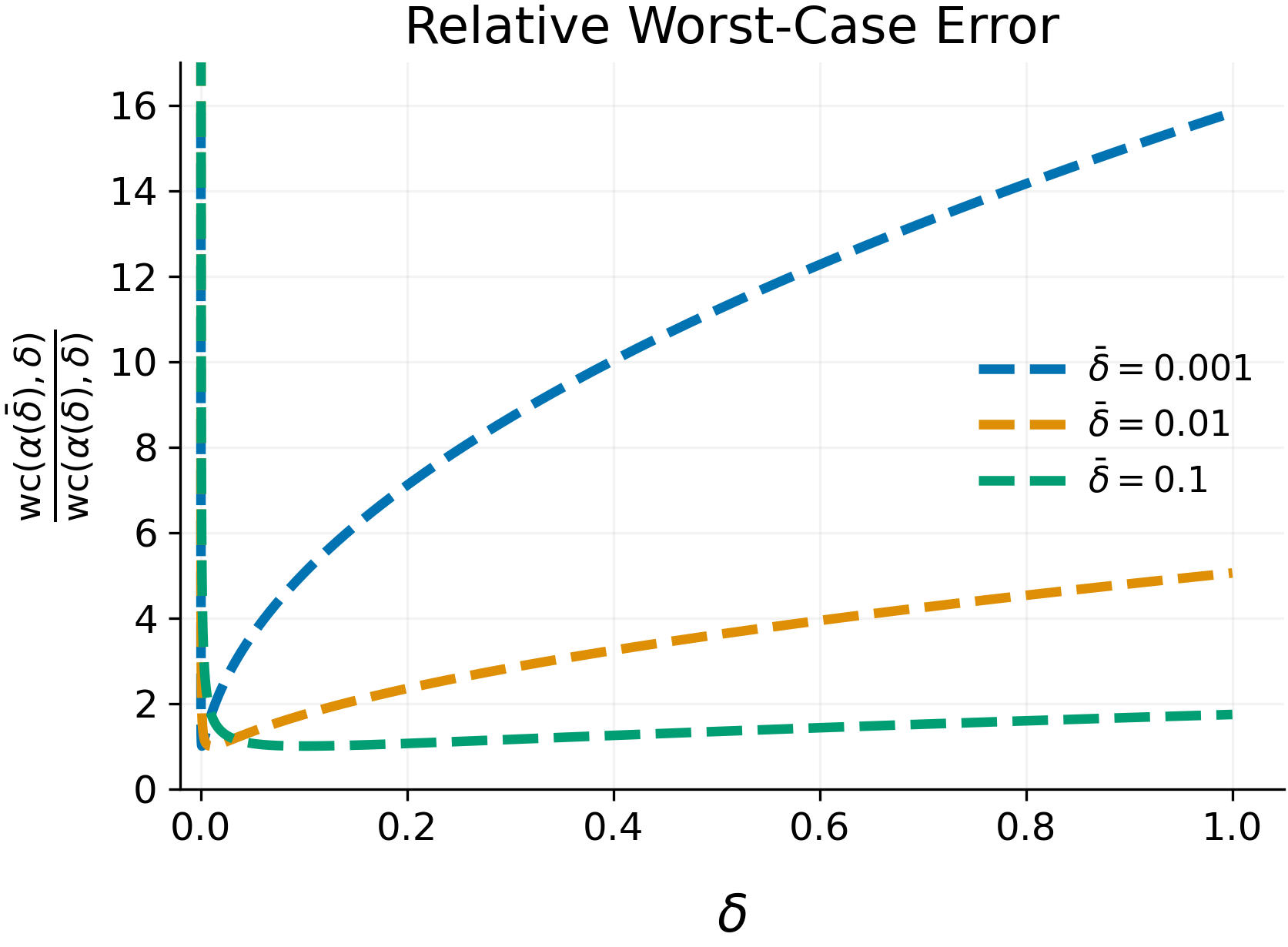}
    \end{subfigure}
    \caption{Analytical wc and relative wc errors as in Corollary~\ref{cor:wc_error} for $A_R$ with $\alpha(\delta) = \nicefrac{\delta}{\rho}$ and ${\rho = 15.74}$.}
    \label{fig:wc_err_radon}
\end{figure}

\begin{figure}[tbp]
  \centering
  \captionsetup[subfigure]{justification=centering}

  \begin{subfigure}[t]{0.28\textwidth}\centering
    \includegraphics[width=\linewidth]{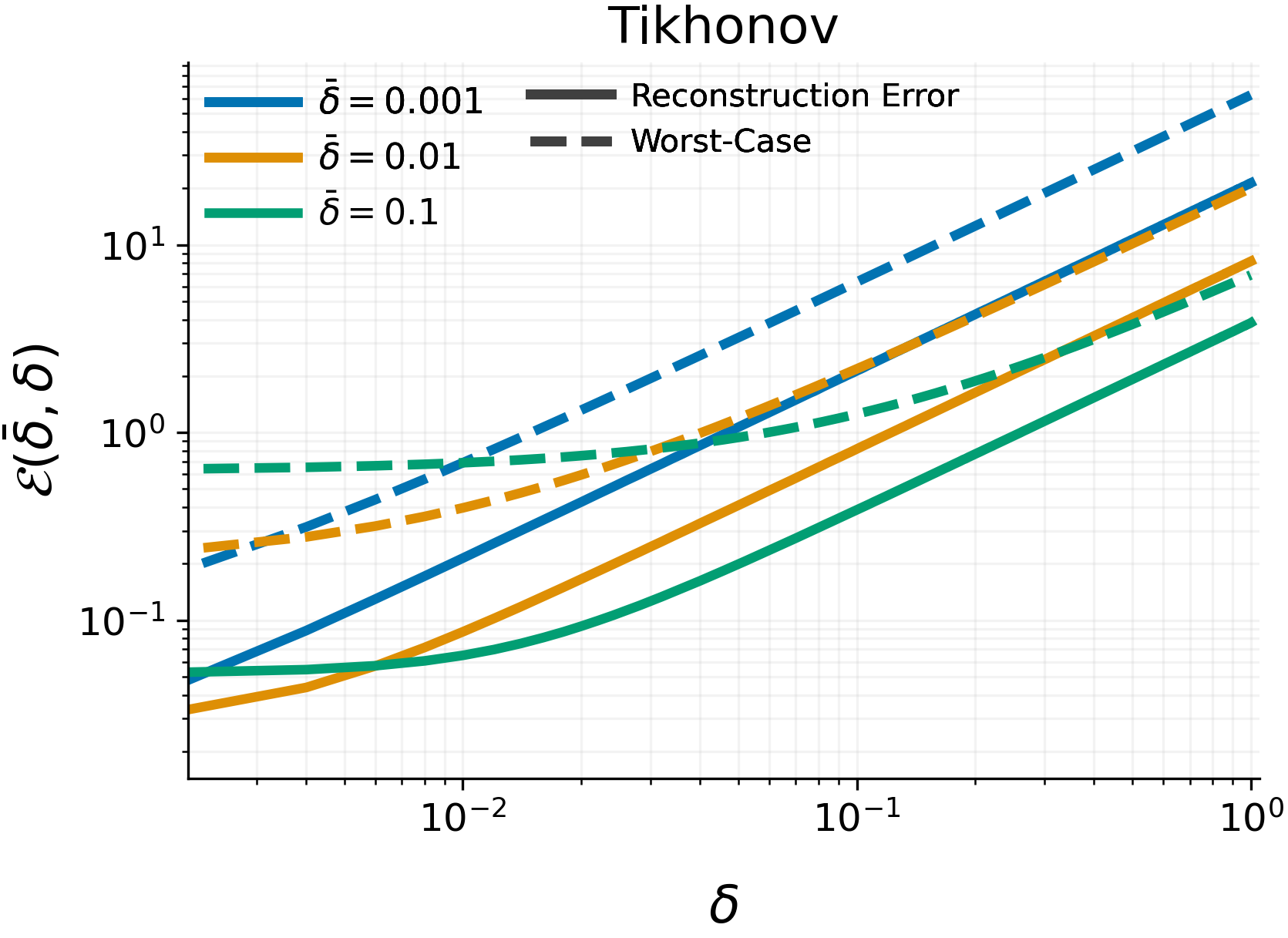}
  \end{subfigure}\hfill
  \begin{subfigure}[t]{0.28\textwidth}\centering
    \includegraphics[width=\linewidth]{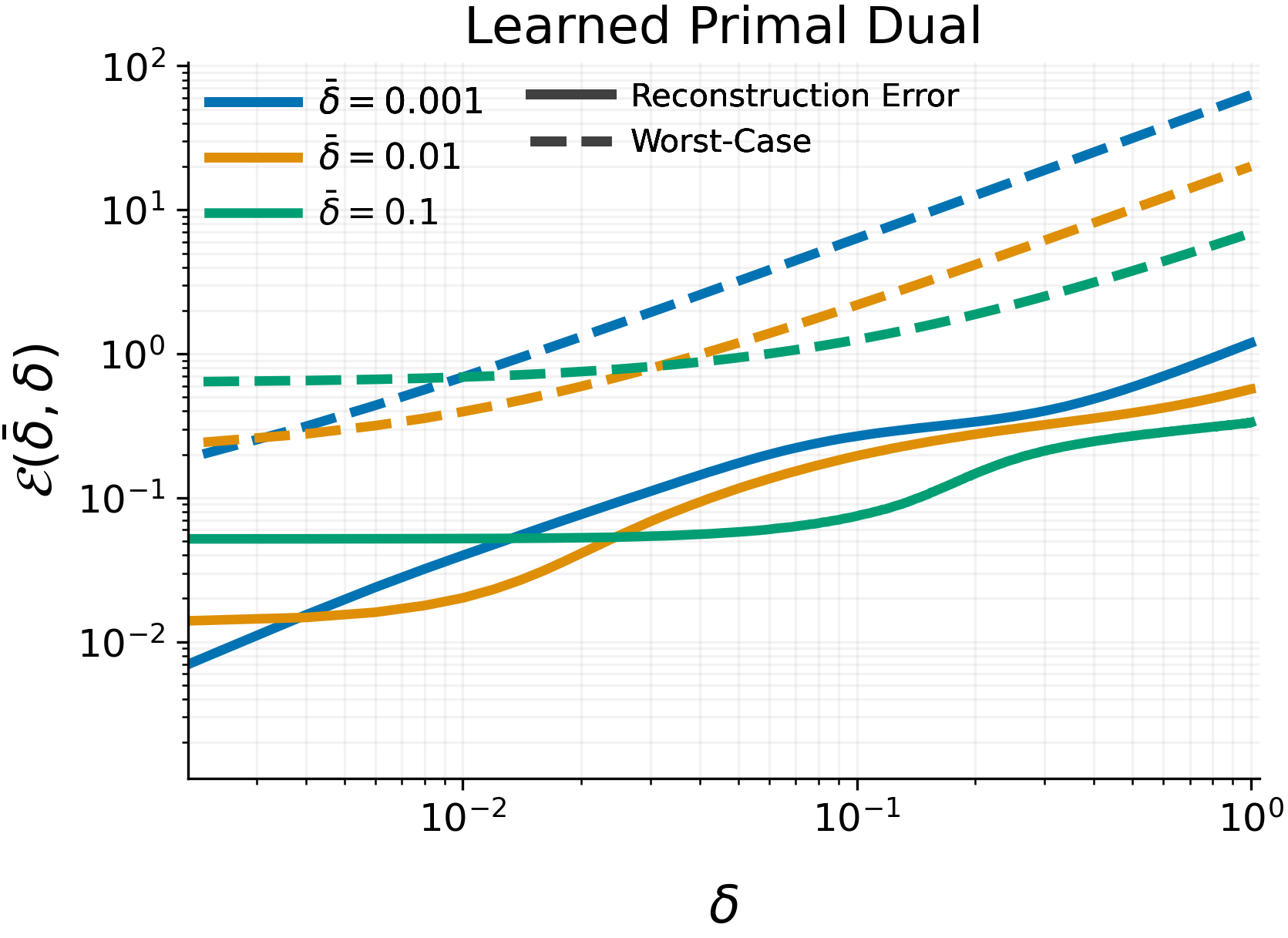}
  \end{subfigure}\hfill
  \begin{subfigure}[t]{0.28\textwidth}\centering
    \includegraphics[width=\linewidth]{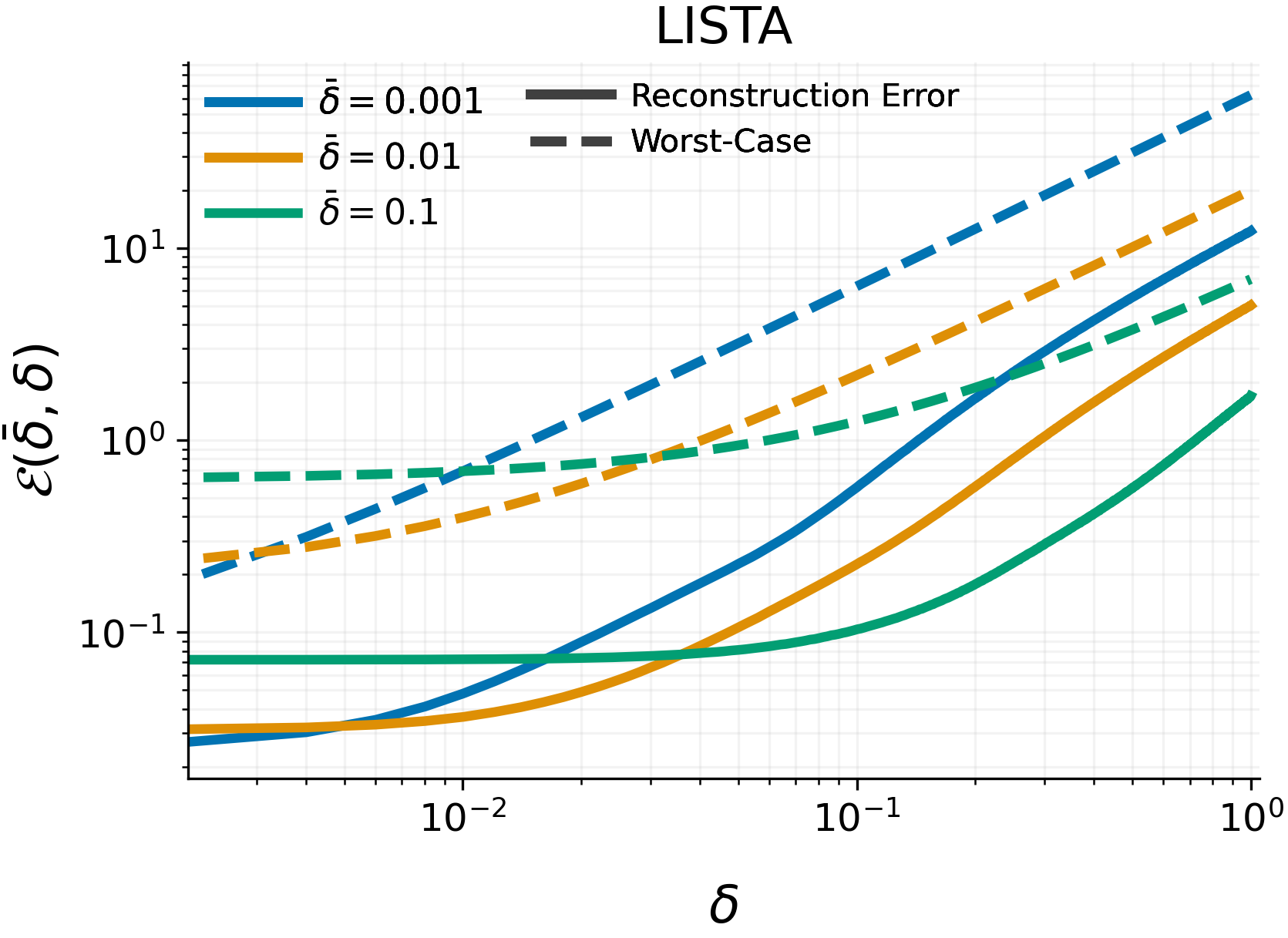}
  \end{subfigure}

  \begin{subfigure}[t]{0.28\textwidth}\centering
    \includegraphics[width=\linewidth]{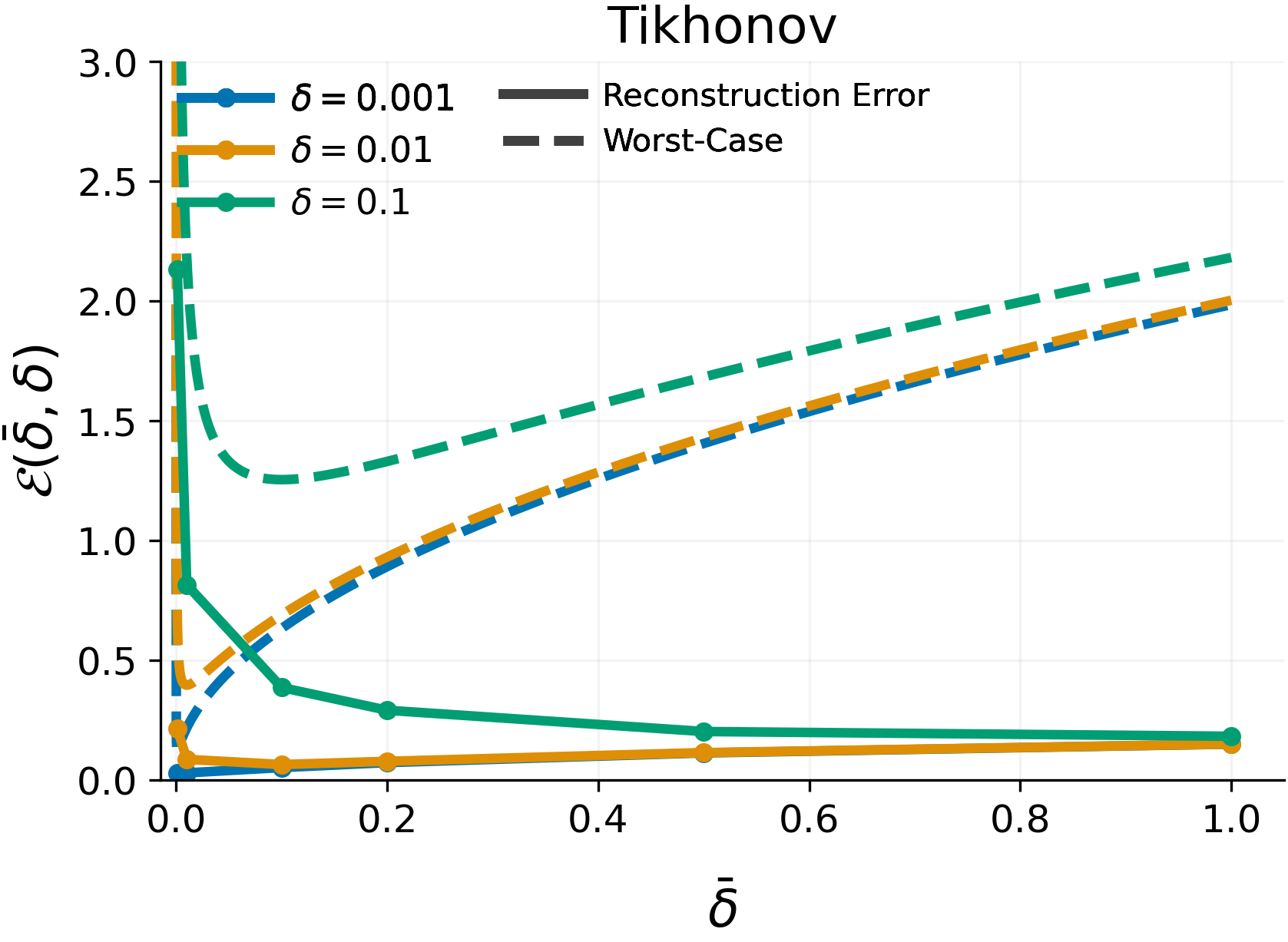}
  \end{subfigure}\hfill
  \begin{subfigure}[t]{0.28\textwidth}\centering
    \includegraphics[width=\linewidth]{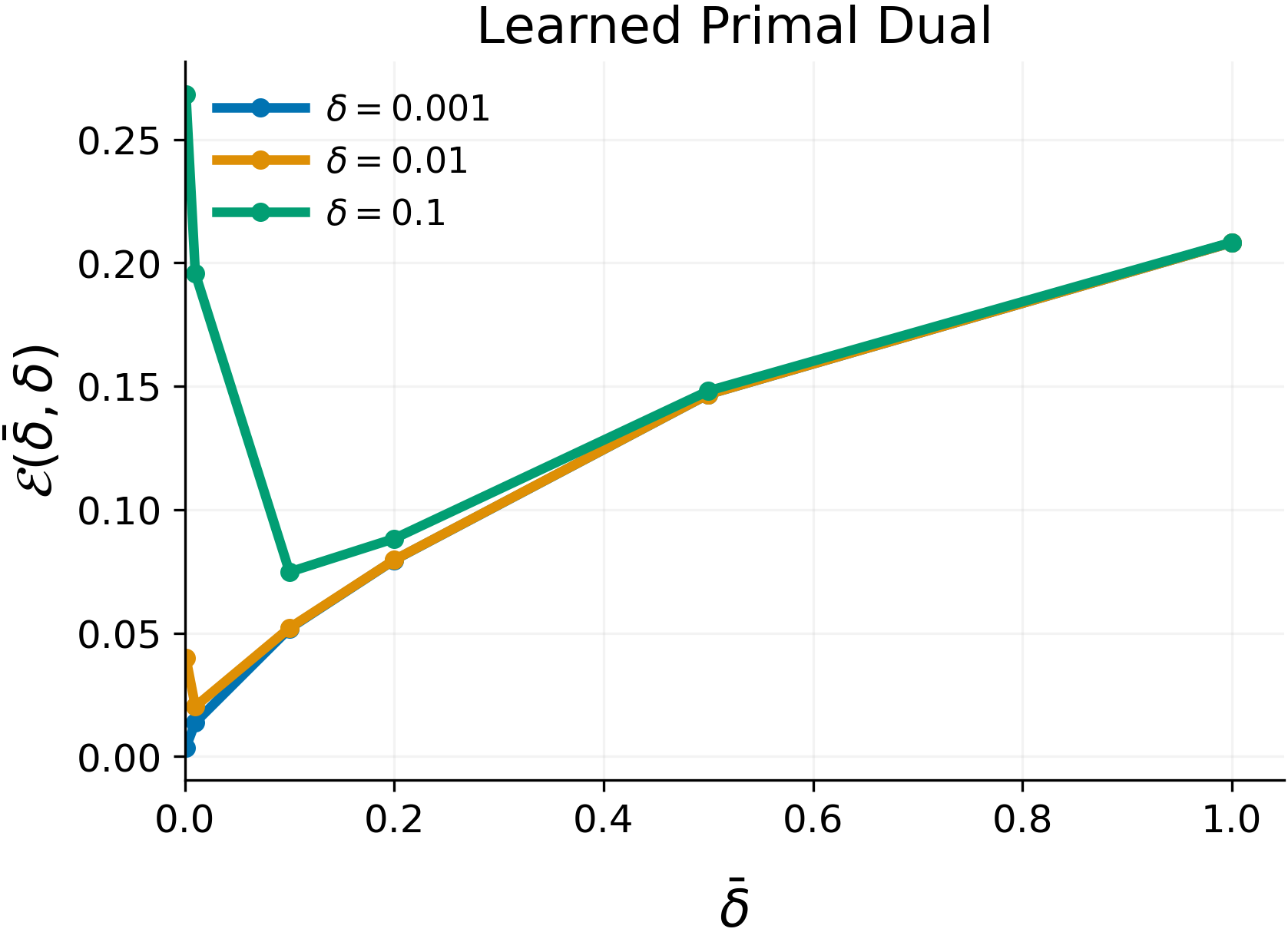}
  \end{subfigure}\hfill
  \begin{subfigure}[t]{0.28\textwidth}\centering
    \includegraphics[width=\linewidth]{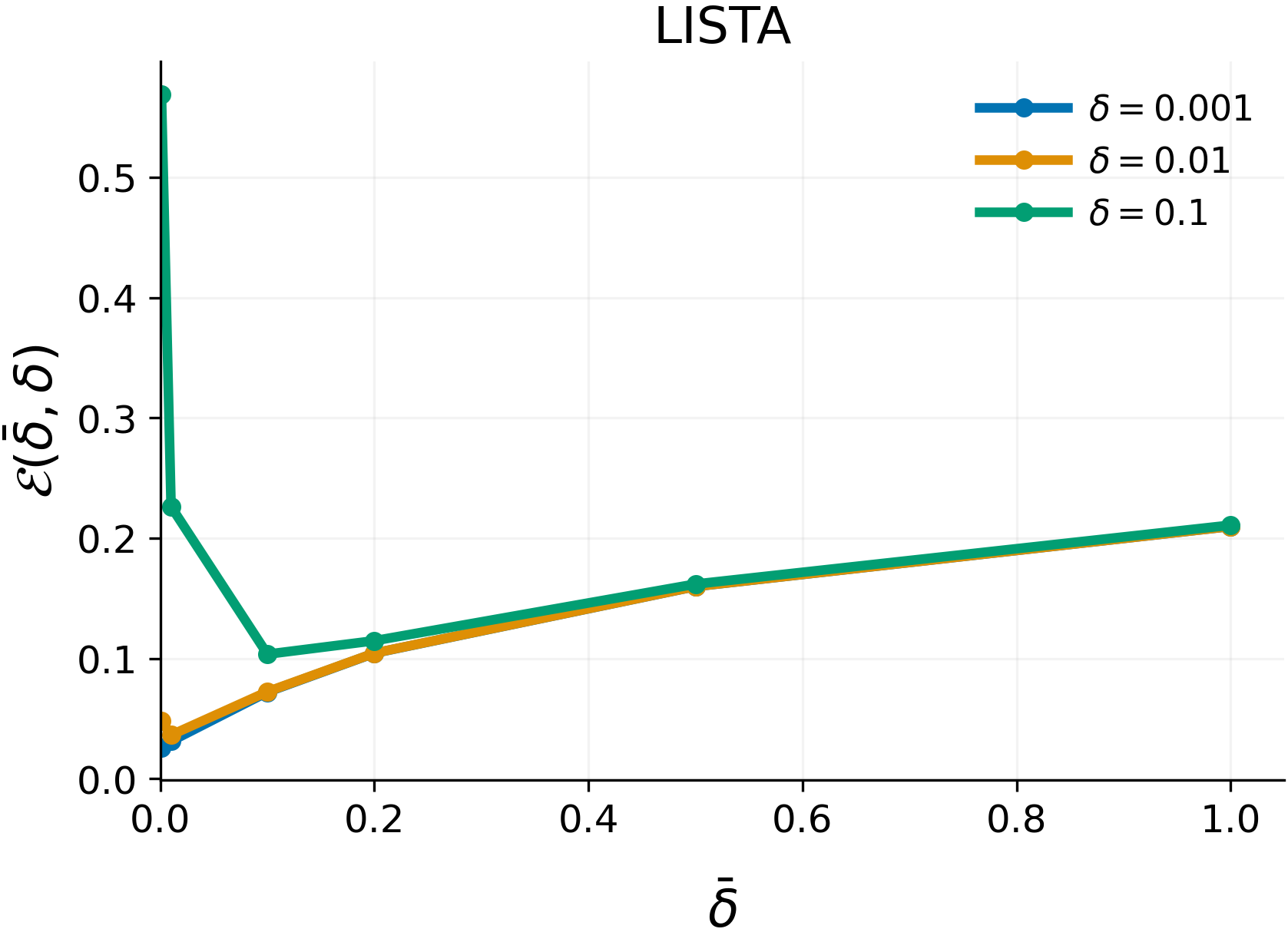}
  \end{subfigure}
  
  \begin{subfigure}[t]{0.28\textwidth}\centering
    \includegraphics[width=\linewidth]{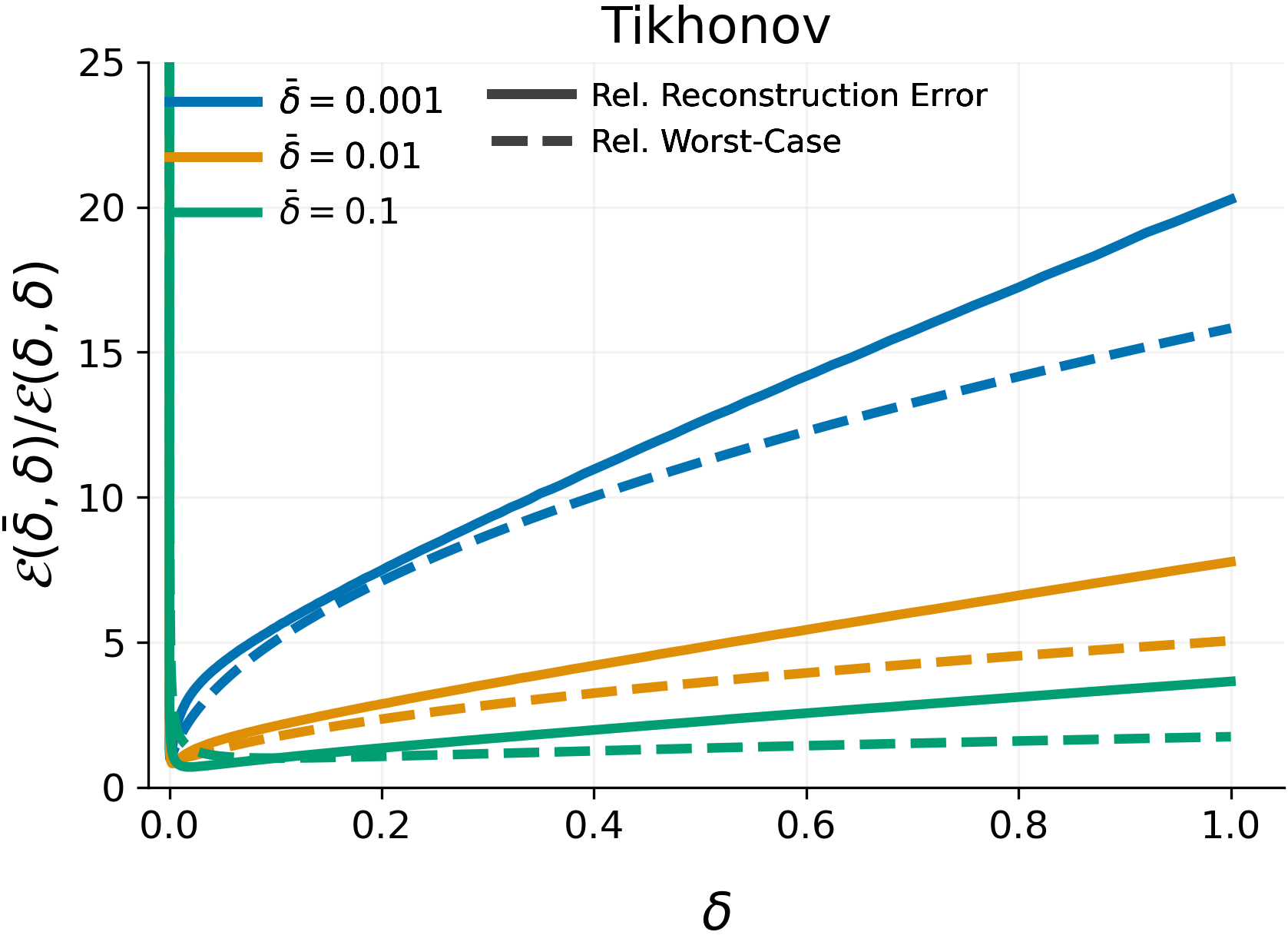}
  \end{subfigure}\hfill
  \begin{subfigure}[t]{0.28\textwidth}\centering
    \includegraphics[width=\linewidth]{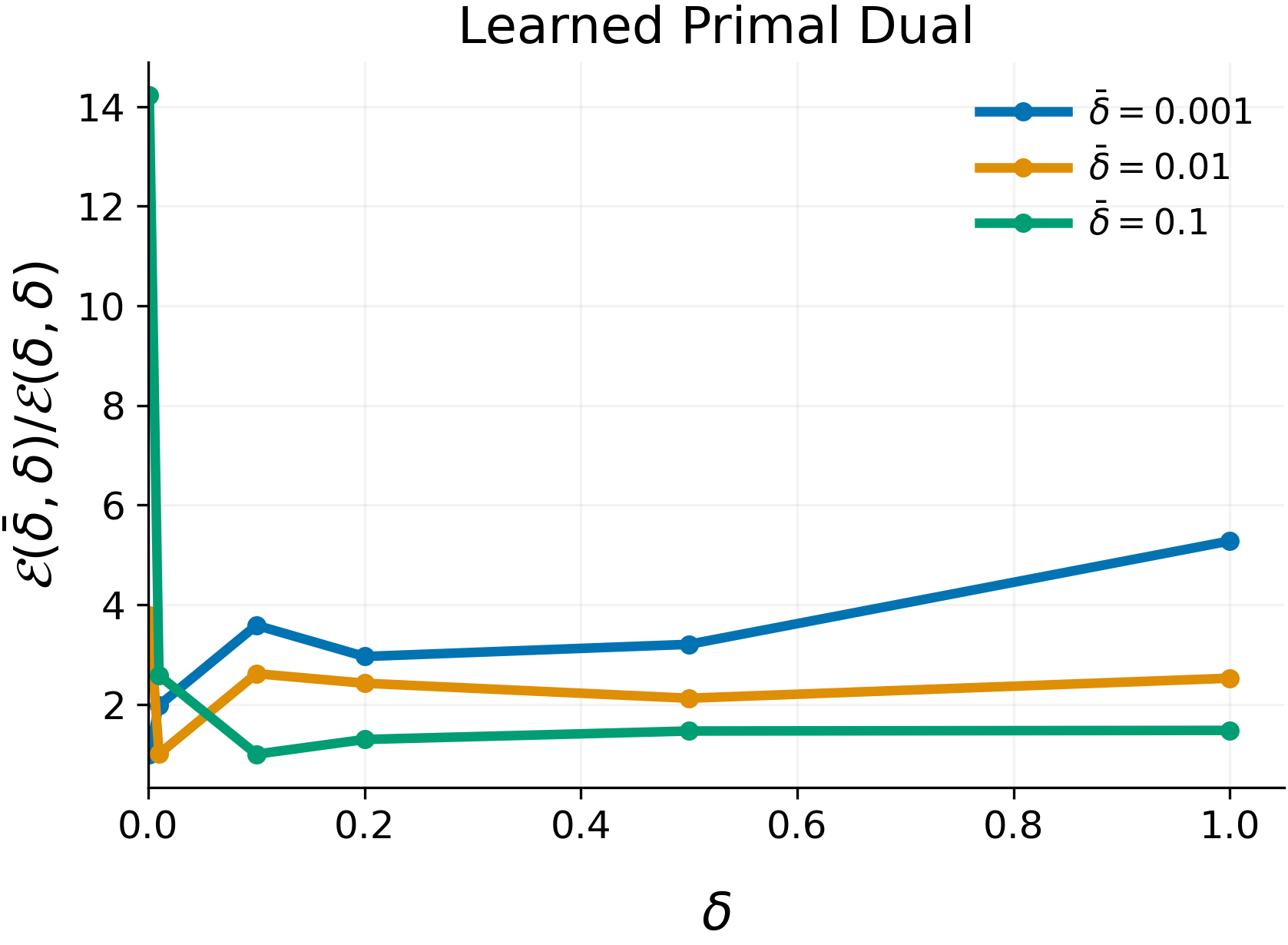}
  \end{subfigure}\hfill
  \begin{subfigure}[t]{0.28\textwidth}\centering
    \includegraphics[width=\linewidth]{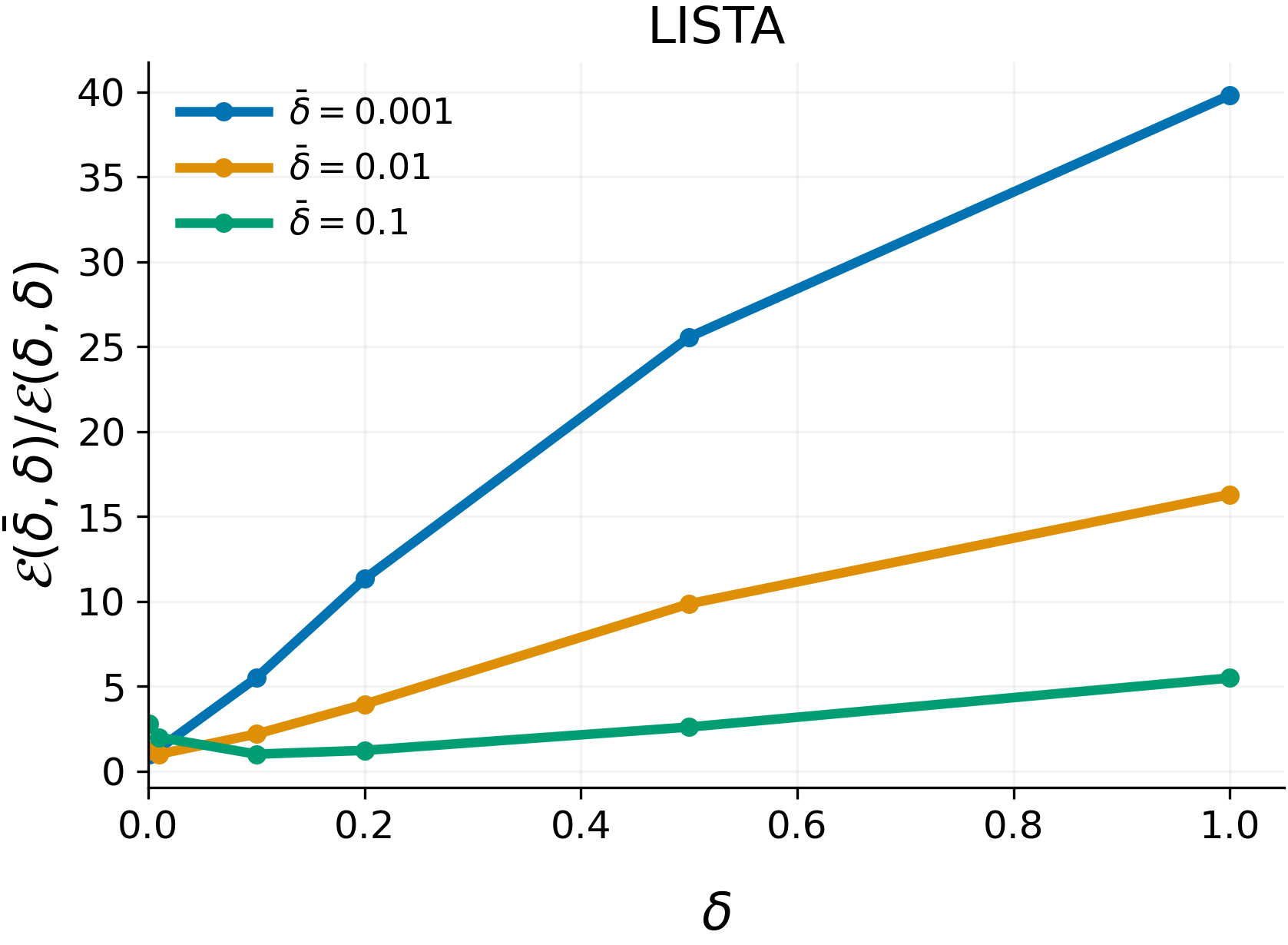}
  \end{subfigure}

  \vspace{-.10cm}
  \noindent\rule{\linewidth}{1pt}
  \vspace{-.10cm}

  \centering
  \captionsetup[subfigure]{justification=centering}

  \begin{subfigure}[t]{0.28\textwidth}\centering
    \includegraphics[width=\linewidth]{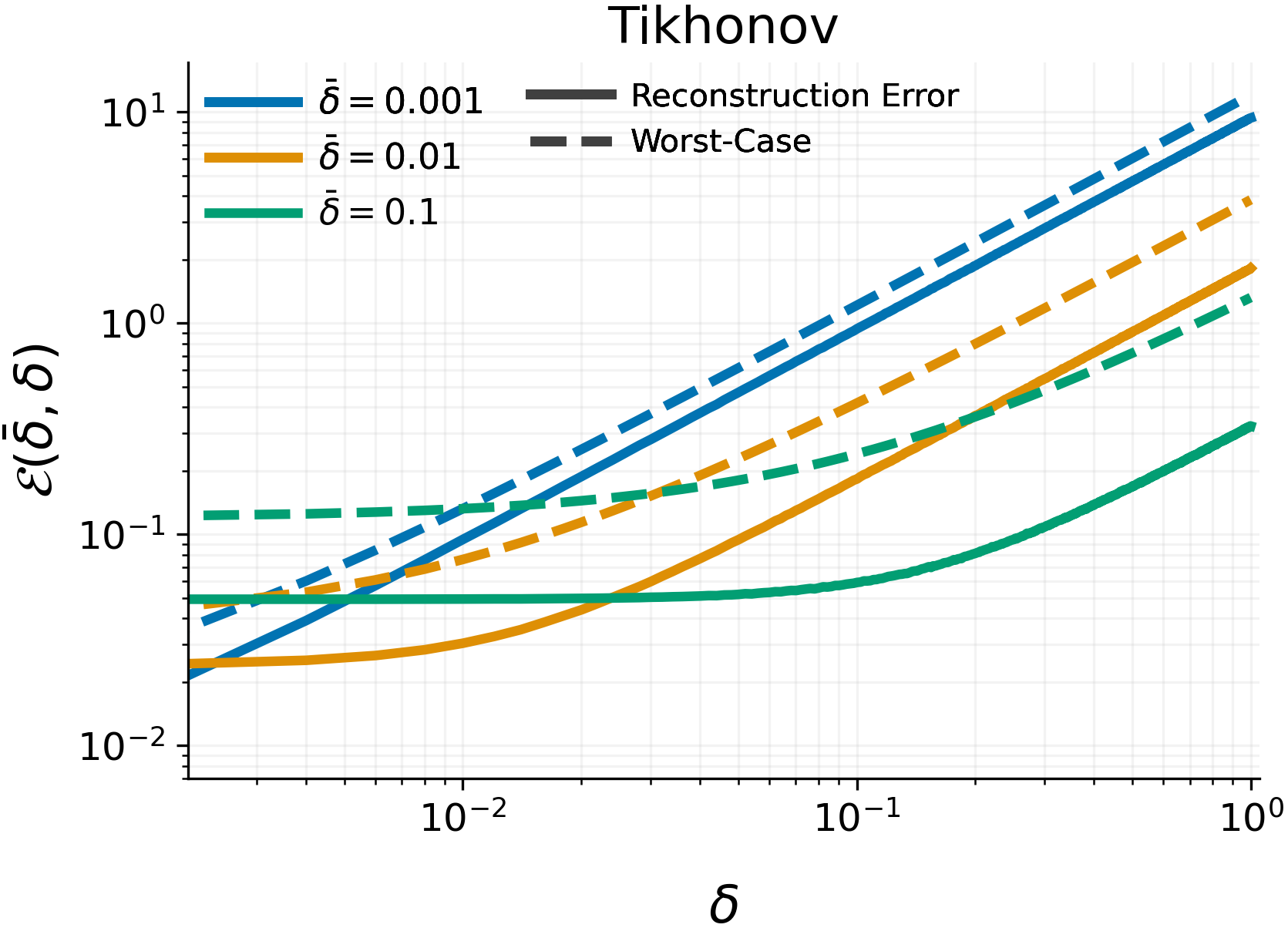}
  \end{subfigure}\hfill
  \begin{subfigure}[t]{0.28\textwidth}\centering
    \includegraphics[width=\linewidth]{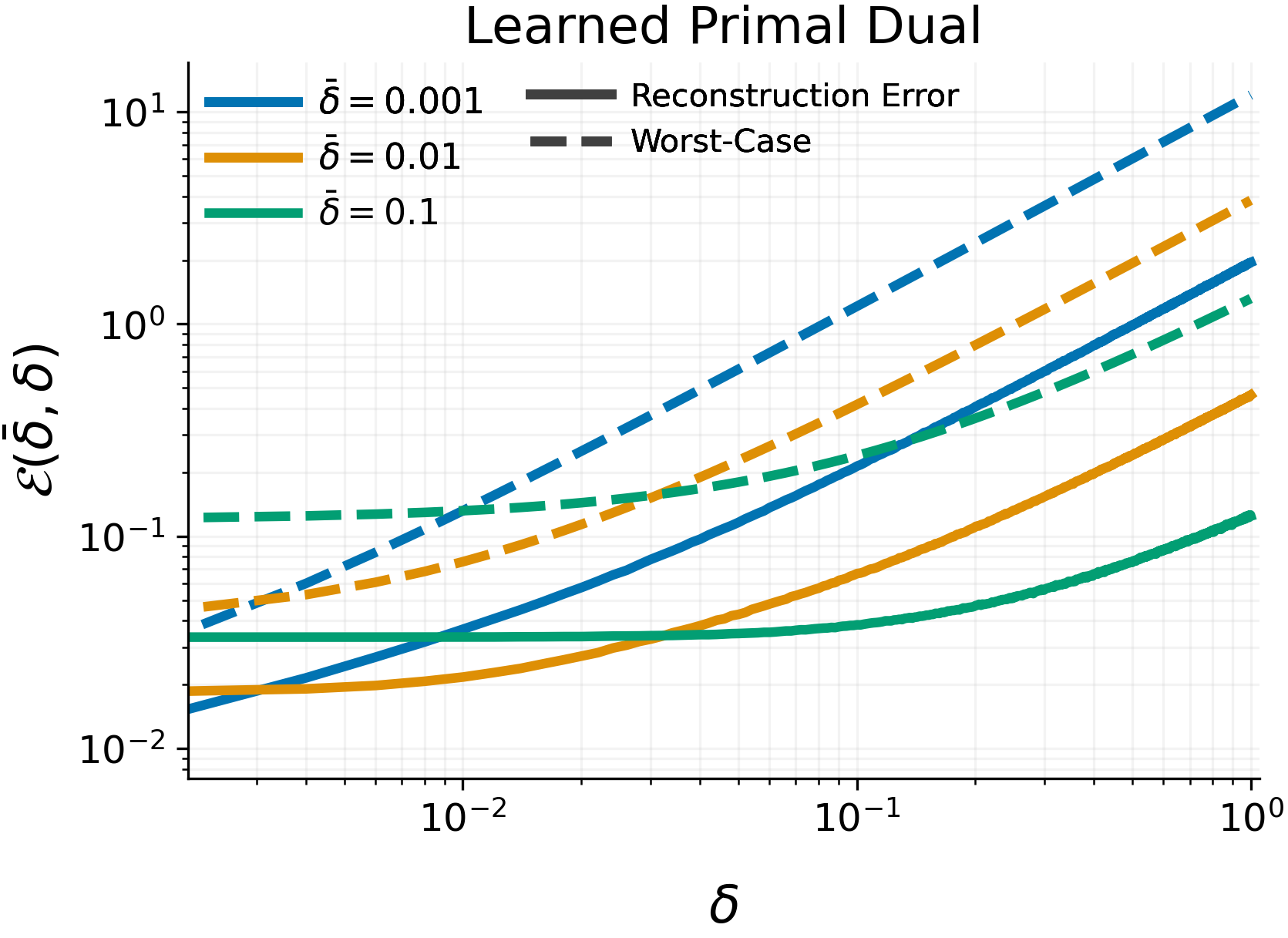}
  \end{subfigure}\hfill
  \begin{subfigure}[t]{0.28\textwidth}\centering
    \includegraphics[width=\linewidth]{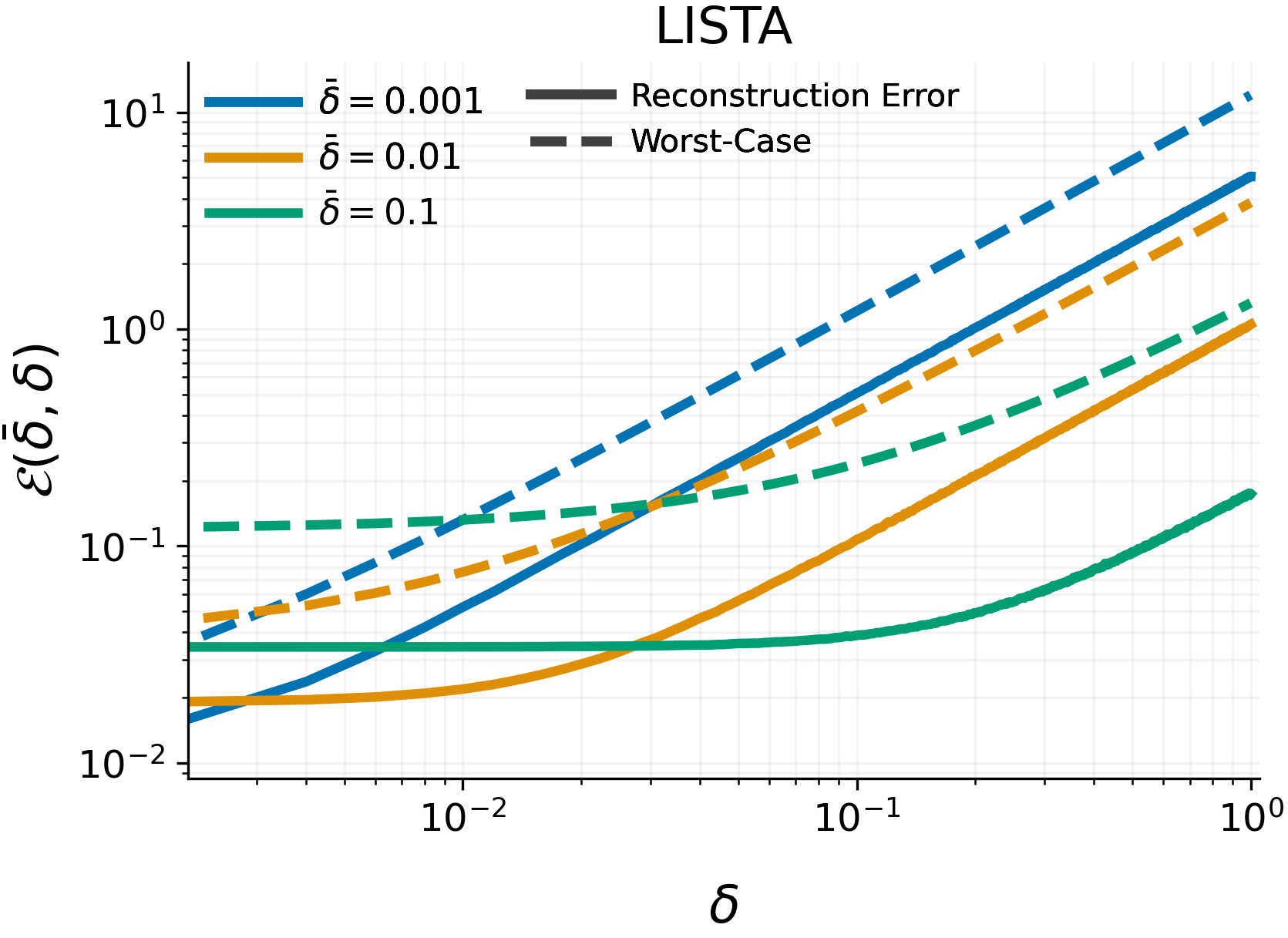}
  \end{subfigure}

  \begin{subfigure}[t]{0.28\textwidth}\centering
    \includegraphics[width=\linewidth]{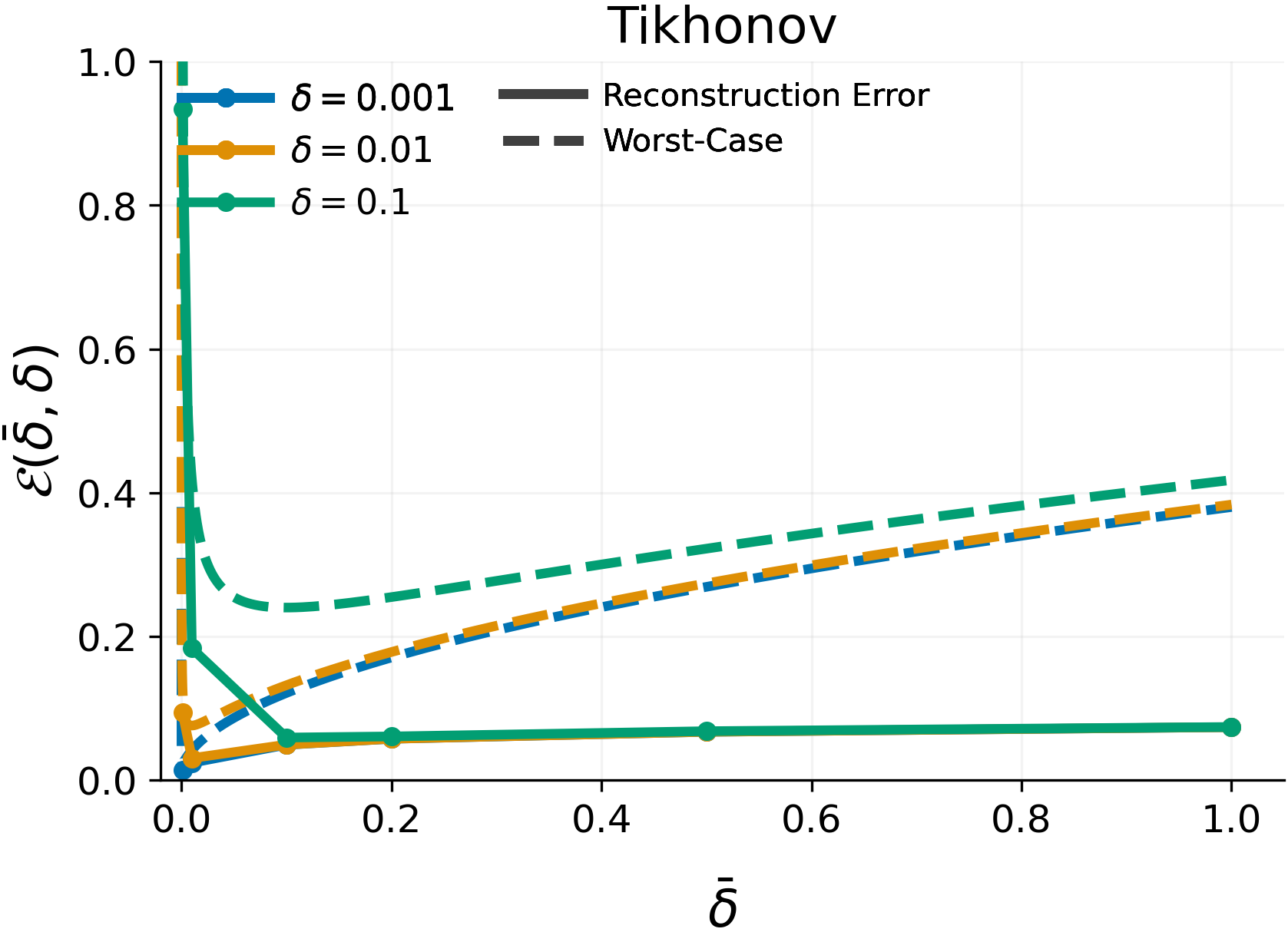}
  \end{subfigure}\hfill
  \begin{subfigure}[t]{0.28\textwidth}\centering
    \includegraphics[width=\linewidth]{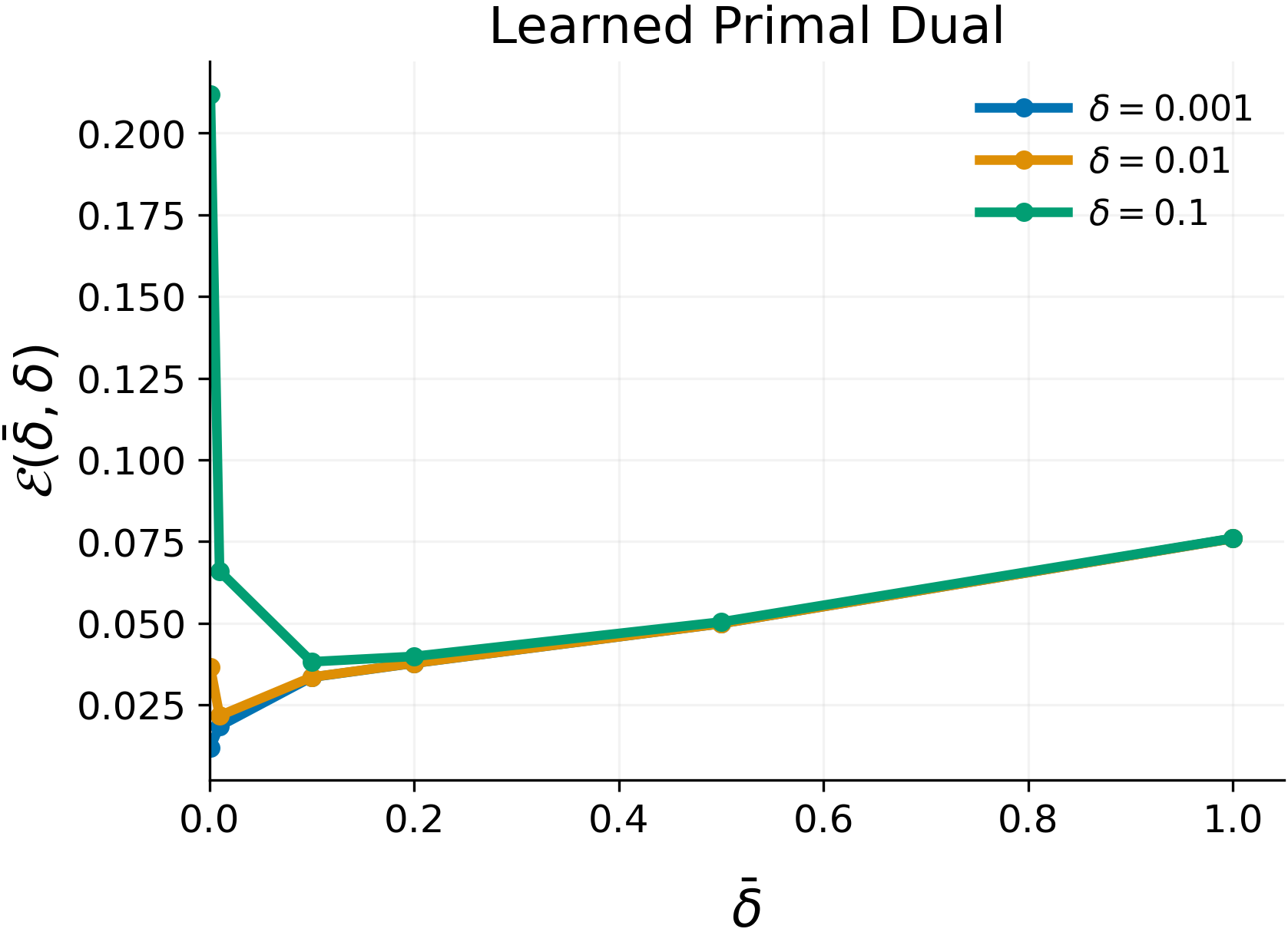}
  \end{subfigure}\hfill
  \begin{subfigure}[t]{0.28\textwidth}\centering
    \includegraphics[width=\linewidth]{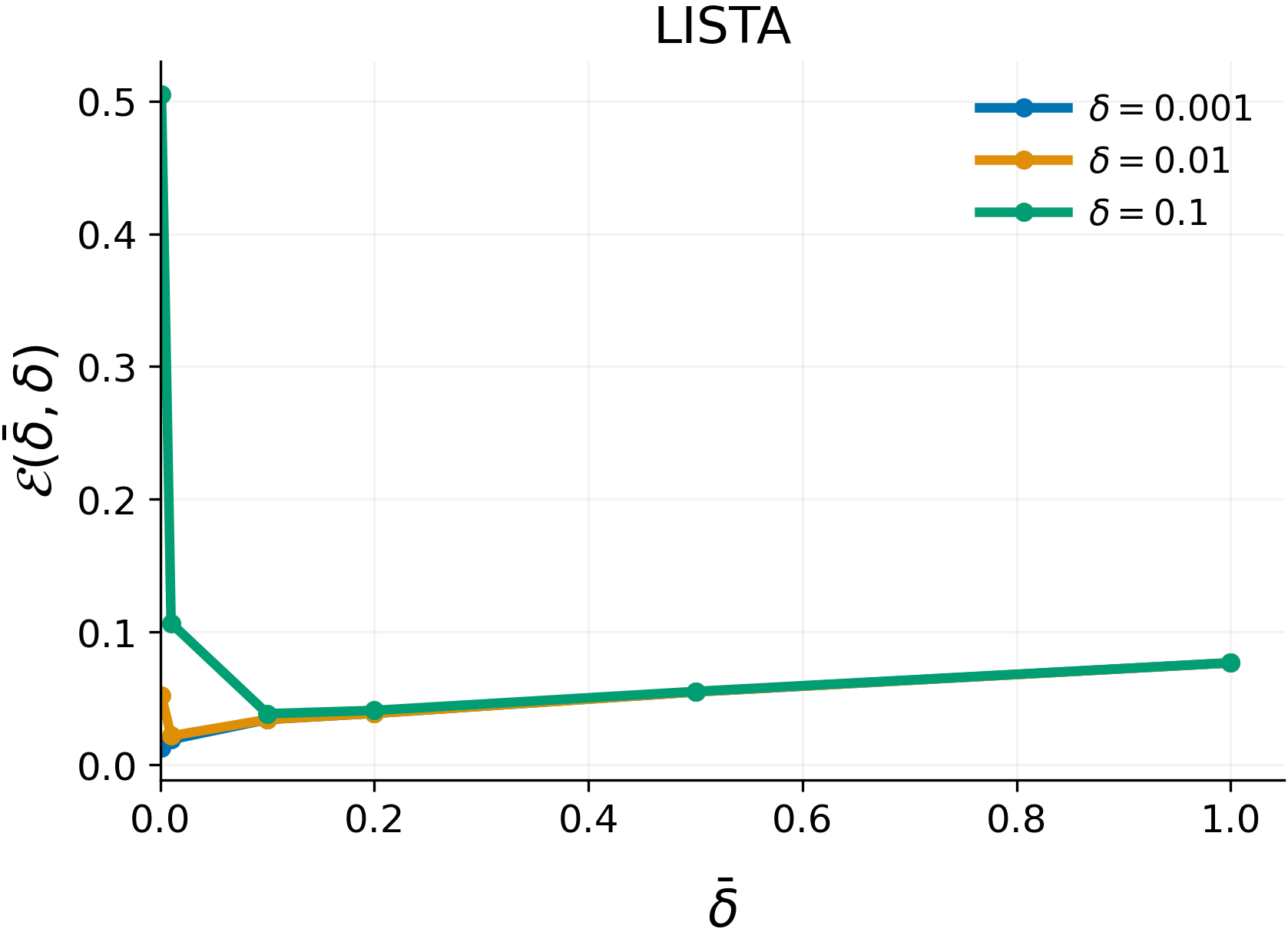}
  \end{subfigure}

  \begin{subfigure}[t]{0.28\textwidth}\centering
    \includegraphics[width=\linewidth]{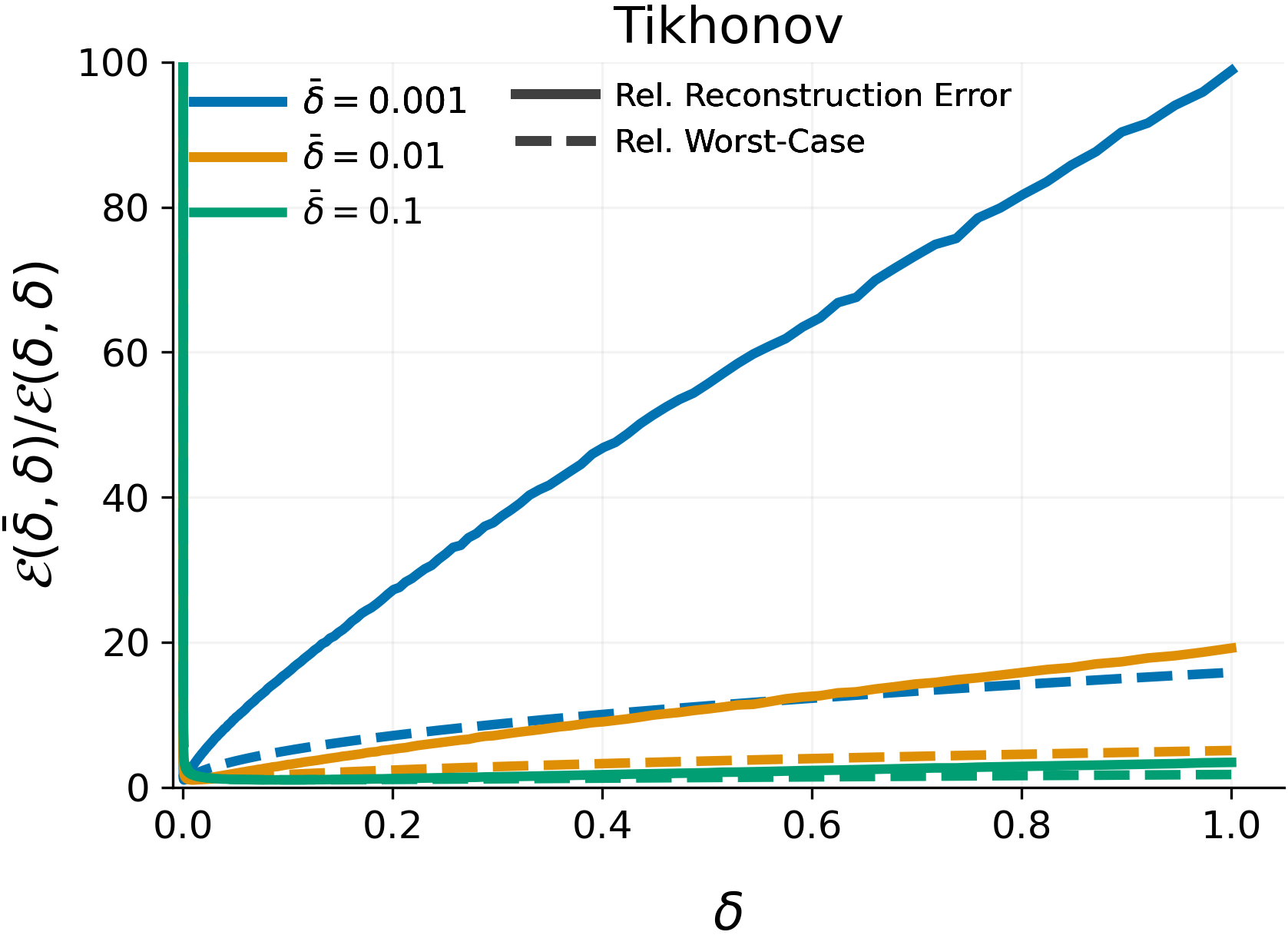}
  \end{subfigure}\hfill
  \begin{subfigure}[t]{0.28\textwidth}\centering
    \includegraphics[width=\linewidth]{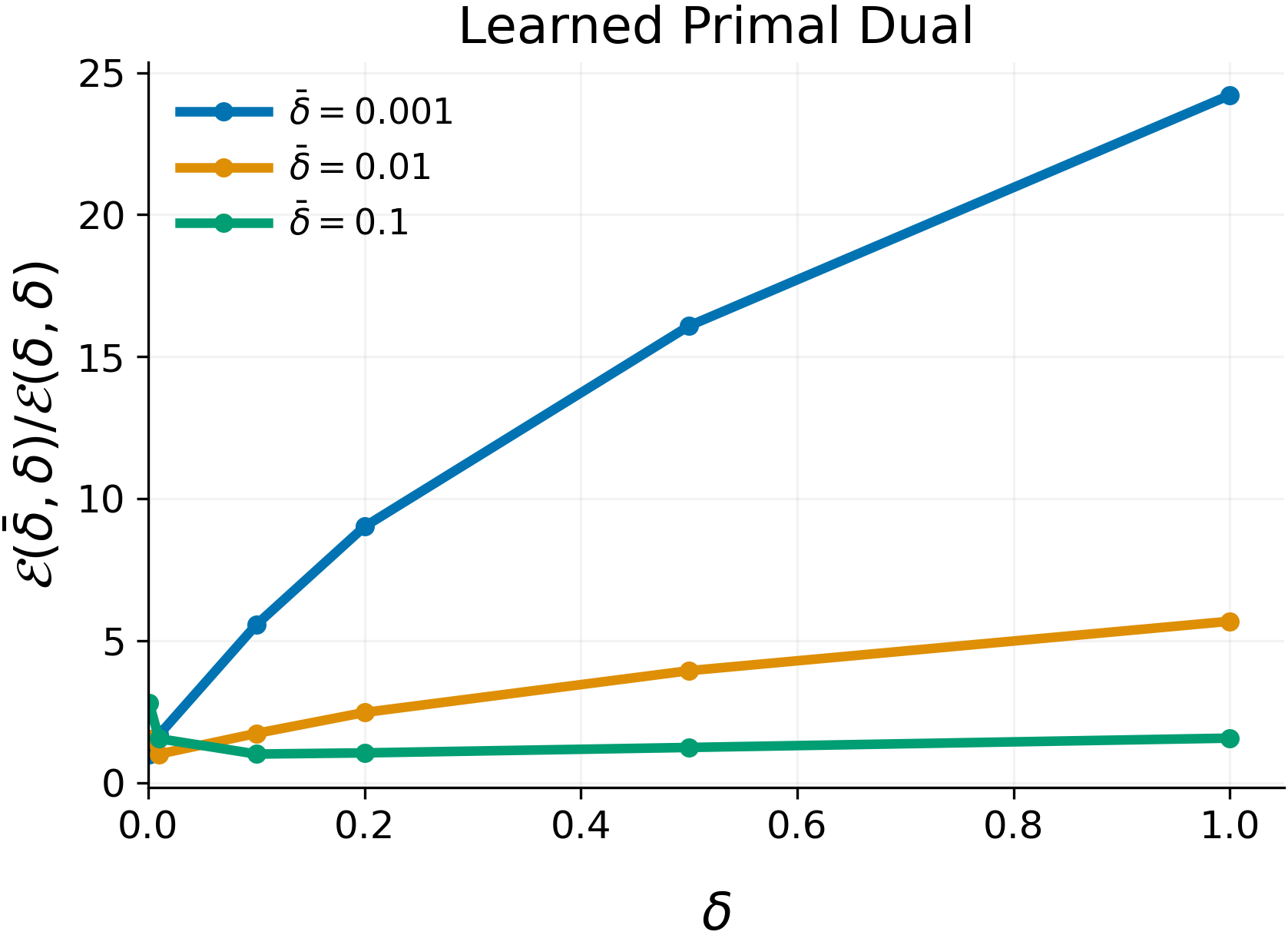}
  \end{subfigure}\hfill
  \begin{subfigure}[t]{0.28\textwidth}\centering
    \includegraphics[width=\linewidth]{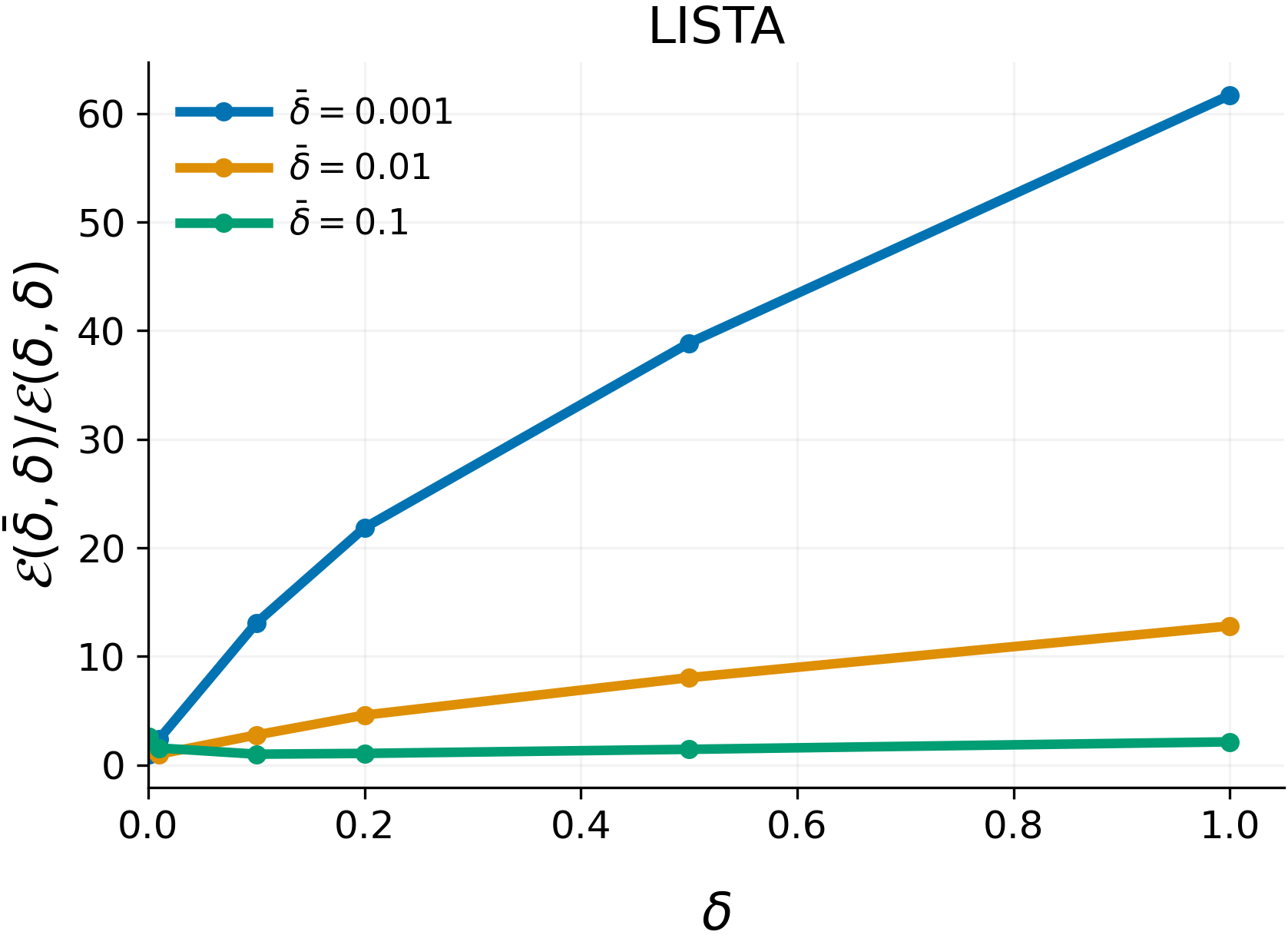}
  \end{subfigure}

  \caption{Top: Reconstruction errors for $A_R$ with Tikhonov regularization ($\rho = 15.74$) \emph{(left)}, LPD \emph{(middle)}, and LISTA \emph{(right)}.
  We plot the errors at fixed regularization level $\bar \delta$ over varying $\delta$~\emph{(upper row)} and over varying regularization levels at fixed $\delta$~\emph{(middle row)} as well as the corresponding relative errors~\emph{(bottom row)}.
  The dashed lines correspond to the analytical bounds as in Corollary~\ref{cor:wc_error}.
  Bottom: The same plots for $A_I$, where we have $\rho = 0.58$.}
  \label{fig:wc_grid_int}
\end{figure}

\section{Worst-case errors on a low-dimensional data manifold}
So far, we analyzed the worst-case error \eqref{eq:DefWC} in the general setting of Hilbert spaces without any assumptions on the underlying data.
In this setting, classical methods are known to yield the optimal worst-case rates and we do not expect major improvements by using learned methods.
Indeed, for our numerical experiments in Section \ref{sec:WorstCaseNumeric}, the learned methods seem to admit the same (optimal) convergence rates.
This might become different if we can exploit the structure of the data manifold.
Hence, we now consider the setting where the ground truth data $x^\dagger$ stem from a low-dimensional linear subspace $X_N \subset X$ of dimension $N$ (both of which are unknown). 
By $A_N = A |_{X_N}$ we denote the restriction of $A$ to $X_N$. 
\subsection{Analysis for Tikhonov regularization}
Again, we first analyze the worst-case error before presenting numerical experiments.
Throughout, we assume the following.
\begin{assumption}\label{ass:2}
\begin{enumerate}[label=(\roman*)]
\item We have $x^\dagger = A^\ast z \in X_N$ with $z\in \mathrm{range}(A_N)$ and $\| z \| \le \rho$.
\item We assume that $A_N$ is only moderately ill-posed in the sense that we have the bound \\ $\|( \alpha I + A^\ast A)^{-1} A^\ast\|_{\mathcal L(\mathrm{range}(A_N), X)}\le CN$ for all $0 < \alpha \leq 1$ and $C > 0$.

\item We consider Tikhonov regularization $T_{\alpha}$ of $A$ with regularization parameter $\alpha$ applied to data $y^\delta$ with $\| y^\delta - y^\dagger \| \le \delta$, where $ \alpha$ is chosen independently of noise level  $\delta$.
To avoid case distinctions, we restrict our analysis to $\alpha \le 1$.
\end{enumerate}
\end{assumption}

\begin{remark}\label{rem:op_bound}
Let us briefly motivate Assumption \ref{ass:2}.
If $A_N$ is injective, then its smallest singular values satisfies $\tilde\sigma_N > 0$, and we obtain the estimate
\begin{equation}
    \|(\alpha I + A^\ast A)^{-1}A^\ast\|_{\mathcal L(\operatorname{range}(A_N),X)}
    \le \tilde\sigma_N^{-1}.
\end{equation}
Hence, any lower bound on $\tilde\sigma_N$ yields a corresponding bound on the restricted Tikhonov operator.
If the singular values decay as $\tilde \sigma_N \sim N^{-1}$, then we obtain precisely (ii) with $C=1$.
In principle, we can use any bound on $\|A^\ast - A_N^\ast\|$ that leads to estimates for $\|( \alpha I + A^\ast A)^{-1}A^\ast\|_{\mathcal L(\mathrm{range}(A_N), X)}$.
\end{remark}

In  view of the subsequent error analysis, we expect that the unknown parameter $N$ serves as an additional intrinsic regularization parameter.
In particular, $\alpha$ should determine the reconstruction quality in a certain range, while $N$ does for the remaining cases.

\begin{theorem}\label{thm:n_dim2}
Under the requirements of Theorem \ref{thm:tikh_error} and Assumption \ref{ass:2}, we obtain for all $ 0 < \alpha\le 1$ that
\begin{equation}\label{eq:extra_bound}
    \|T_\alpha y^\delta - x^\dagger\|
     \le  \frac{\delta}{2\sqrt{\alpha}}
+ \begin{cases}
       \frac{\sqrt{\alpha}}{2}\rho &\text{if } 1 / (2CN) < \sqrt \alpha \leq 1, \\
    \alpha C N \rho &\text{if } \sqrt \alpha \le 1 / (2CN).
    \end{cases}   
\end{equation}
\end{theorem}

\begin{proof}
From the proof of Theorem \ref{thm:tikh_error}, the data error can be bounded by $\|T_\alpha (y^\delta-y^\dagger)\| \leq \delta/(2\sqrt{\alpha})$, see \eqref{eq:EstDataError}, and the approximation error can be bounded by $\| T_\alpha y^\dagger - x^\dagger\| \leq \sqrt{\alpha}\rho/2$, see \eqref{eq:EstApproxError}.
For the approximation error, we obtain a sharper bound for small values of $\alpha$.
More precisely, we have
\begin{equation}
    T_\alpha y^\dagger = (A^\ast A + \alpha I)^{-1}A^\ast A x^\dagger
\quad \text{and} \quad T_\alpha y^\dagger - x^\dagger = - \alpha (\alpha I + A^\ast A)^{-1}A^\ast z.
\end{equation}
Due to the assumption  $\|( \alpha I + A^\ast A)^{-1}A^\ast \|_{\mathcal L(\mathrm{range}(A_N), X)}\le CN$ we obtain that
\begin{equation}
    \| T_\alpha y^\dagger - x^\dagger \| \le \alpha CN \rho,
\end{equation}
so that the approximation error satisfies
\begin{equation}
\|T_\alpha y^\dagger - x^\dagger\|
 \le 
 \begin{cases}
       \frac{1}{2}\sqrt{\alpha}\rho &\text{if } 1 / (2CN) < \sqrt \alpha \leq 1, \\
    \alpha   C N \rho &\text{if } \sqrt \alpha \le 1 / (2CN).
    \end{cases}
\end{equation}
In view of the error decomposition \eqref{eq:ErrorDecompose}, this concludes the proof.
\end{proof}
In summary, the classical worst-case bounds for Tikhonov regularization with small $\alpha$ from Theorem~\ref{thm:tikh_error} are overly pessimistic.
In fact, if the admissible solutions are confined to a low-dimensional subspace $X_N$ with a specific rate for the condition number of the restricted operator $A_N$, the effective ill-posedness is much weaker.
Then, according to Corollary~\ref{cor:wc_error}, the sensitivity of the worst-case error to a misspecified regularization parameter $\alpha$ is also reduced.

\begin{remark}
Theorem \ref{thm:n_dim2} shows that the worst-case error as a function of $\alpha$ changes its analytic  at  $\sqrt \alpha \sim 1/(2CN)$. The function is continuous at this point, while its derivative is not. 
In principle, this allows to determine the unknown intrinsic dimension $N$ by computing several Tikhonov reconstructions with varying $\alpha$.
Evaluating \eqref{eq:extra_bound} requires the knowledge of $x^\dagger$, which however can by substituted by a high quality reference reconstruction.
\end{remark}

\subsection{Numerical results for unconstrained reconstruction}

We present numerical experiments in analogy to Section~\ref{sec:WorstCaseNumeric} for the integral operator $A_I \in \R^{n\times n}$ with $n=50$.
For this, we replace the data used in Section~\ref{sec:WorstCaseNumeric} by data that lives in a finite-dimensional linear subspace $X_N \subseteq X = \R^n$ and fulfills Assumption~\ref{ass:2}.
We first sample $N$ left singular vectors $(u_{k_j})_{j=1}^N$ of $A$ uniformly and then generate data by sampling 
\begin{equation}
    z_i = \sum_{j=1}^N d_j u_{k_j} \quad \text{ with } d_j \sim \mathcal{U}[-1,1].
\end{equation}
This is the natural restriction of~\eqref{eq:zdata_gen} to a subspace.
Note that 
\begin{equation}
    x_i \in X_N \coloneqq \{ v_{k_1}, \dots, v_{k_N}\} \quad \text{for all } i = 1,\hdots,L_\mathrm{train}.
\end{equation}
By construction, it holds $z_i \in \mathrm{range}(A_N)$.
For the experiments, we choose $N=8$.
Since the singular vectors of $A$ scale approximately as $n^{-1}$ and we consider the SVD basis, Assumption~\ref{ass:2}(ii) is also satisfied according to Remark~\ref{rem:op_bound}.

We again compare Tikhonov reconstruction, for which the worst-case error is given in Theorem~\ref{thm:n_dim2}, against the learned schemes LPD and LISTA in Figure~\ref{fig:wc_grid_int_nd}.
In addition to the characteristics described in Section~\ref{sec:WorstCaseNumeric} for general data, the learned methods now yield a significantly better reconstruction error compared to Tikhonov regularization.
This can be explained by the fact that the signals lie in the 8-dimensional subspace $X_{8} \subset \R^{50}$.
Hence, LISTA and LPD can learn reconstruction maps that are effectively adapted to $A_I^\dagger$ on this signal class while suppressing irrelevant ambient directions. 
Classical Tikhonov regularization, by contrast, remains a global linear filter that cannot exploit the low-dimensional structure of the data.
This effect is much weaker in the MNIST–Radon setting, where the data is only approximately low-dimensional and the inverse problem is considerably more ill-posed.
For larger $\bar\delta$, the measurements are more degraded, so that the learned reconstructions appear to benefit less from the underlying structure. Note that the second case of the error bound in~\eqref{eq:extra_bound} is invisible in the plots.
For $N=8$, this regime is restricted to  $\alpha \lesssim \nicefrac{1}{(16C)^2} \approx \nicefrac{1}{256}$ for $C \approx 1$.
Thus, it only affects a narrow parameter range.
Moreover, the corresponding contribution only induces slight vertical shift of the error curves when plotted against $\delta$, which becomes negligible compared to the data error.

\begin{figure}[tbh]
  \centering
  \captionsetup[subfigure]{justification=centering}

  \begin{subfigure}[t]{0.32\textwidth}\centering
    \includegraphics[width=\linewidth]{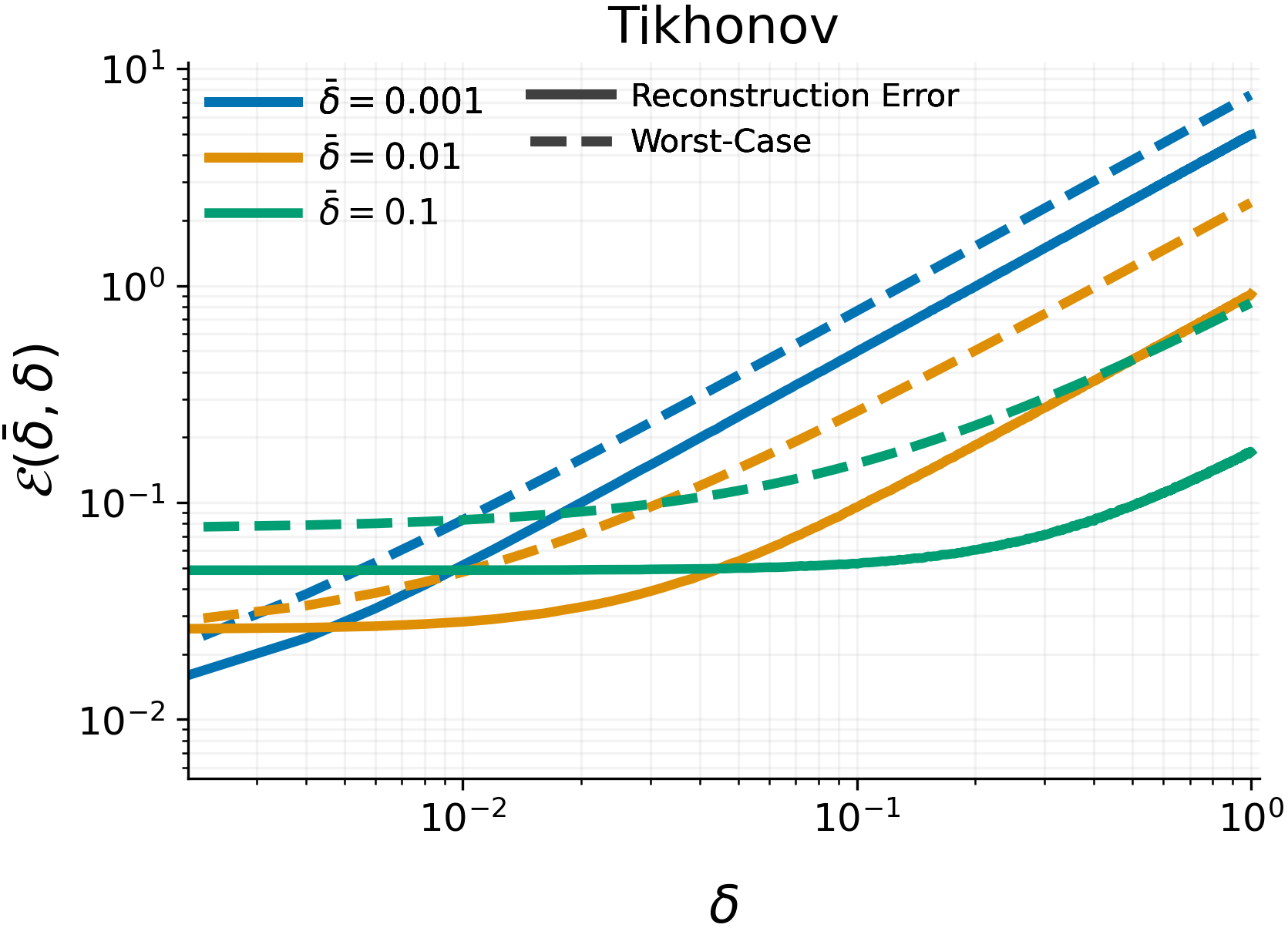}
  \end{subfigure}\hfill
  \begin{subfigure}[t]{0.32\textwidth}\centering
    \includegraphics[width=\linewidth]{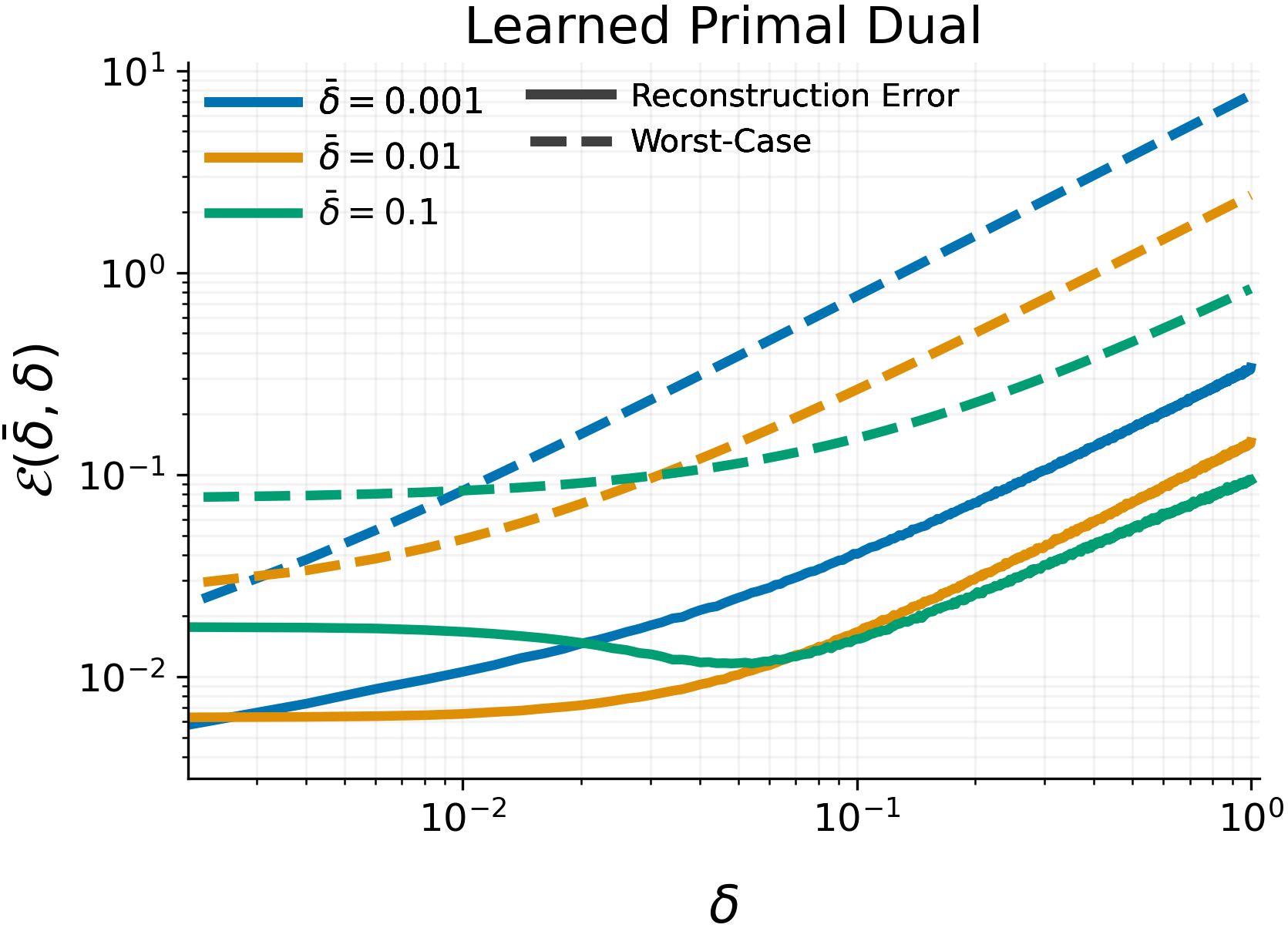}
  \end{subfigure}\hfill
  \begin{subfigure}[t]{0.32\textwidth}\centering
    \includegraphics[width=\linewidth]{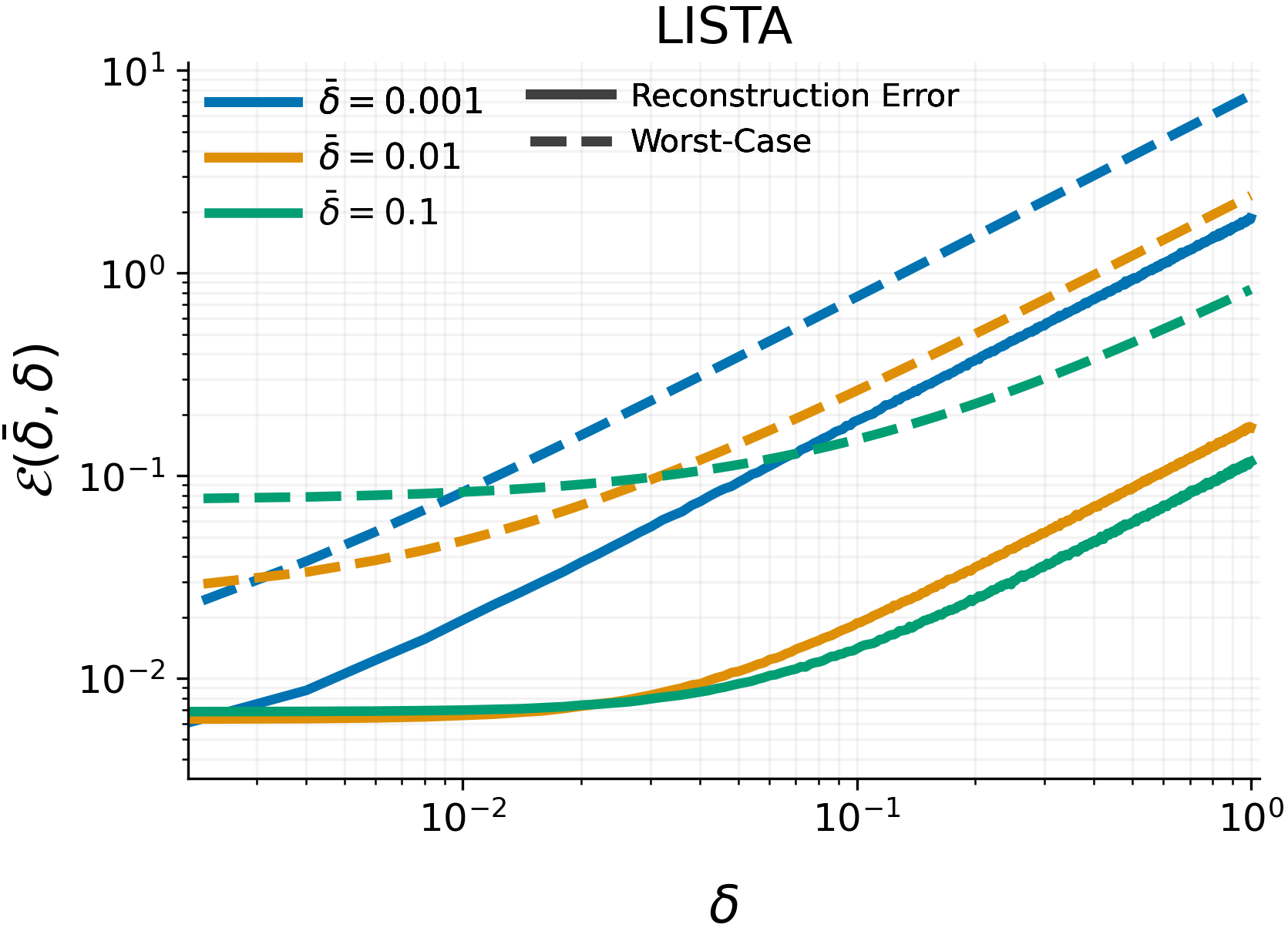}
  \end{subfigure}

  \begin{subfigure}[t]{0.32\textwidth}\centering
    \includegraphics[width=\linewidth]{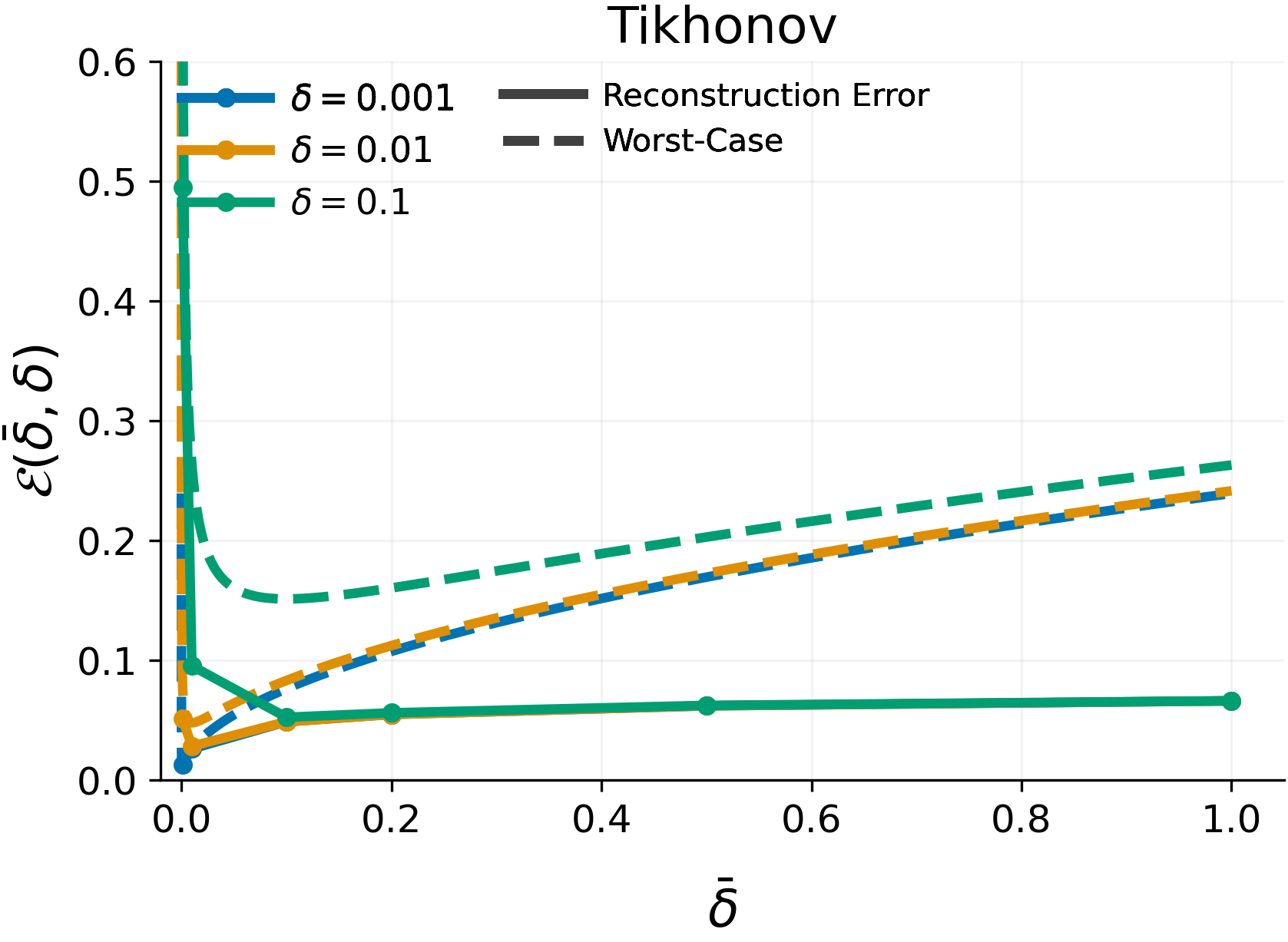}
  \end{subfigure}\hfill
  \begin{subfigure}[t]{0.32\textwidth}\centering
    \includegraphics[width=\linewidth]{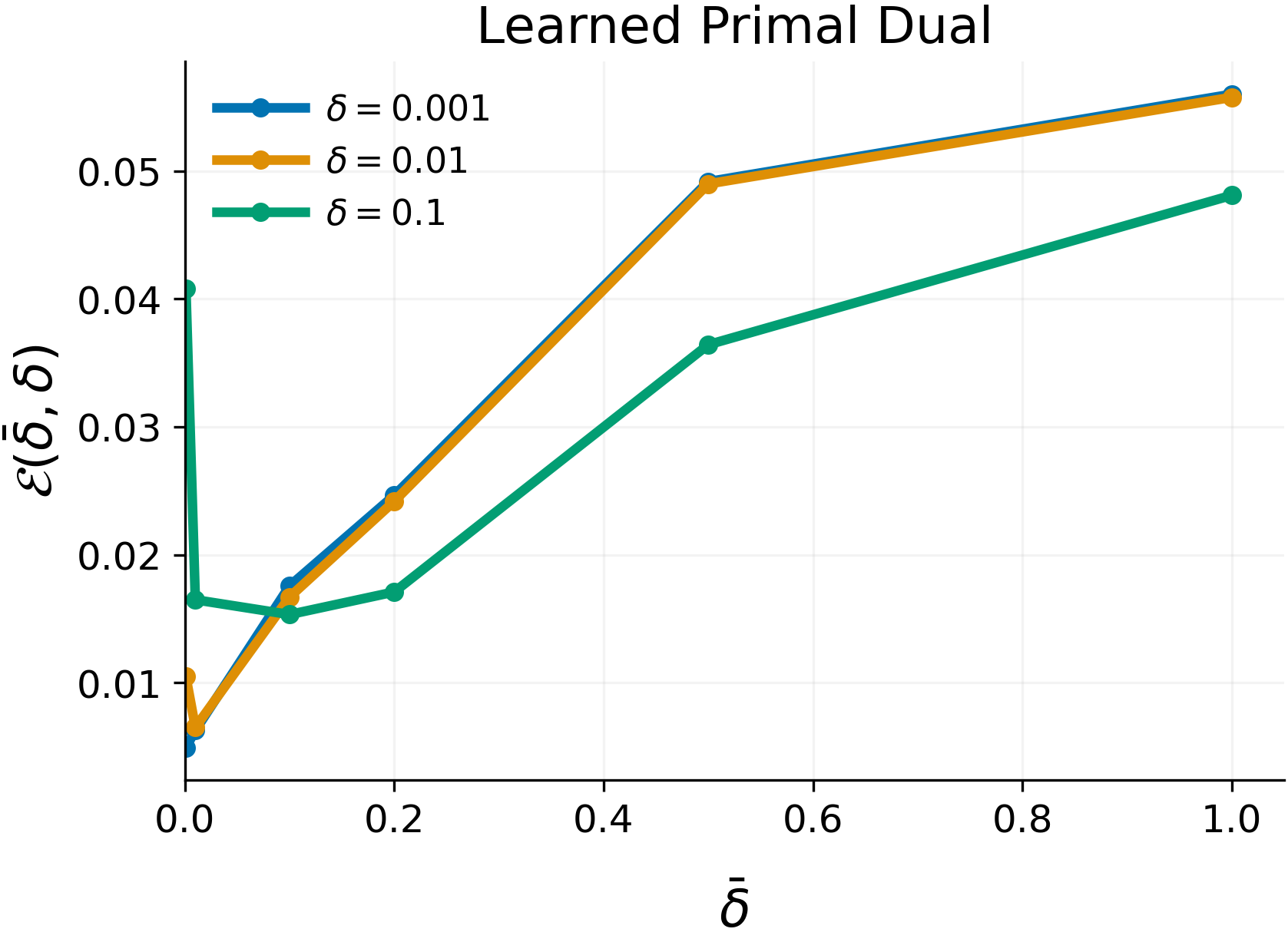}
  \end{subfigure}\hfill
  \begin{subfigure}[t]{0.32\textwidth}\centering
    \includegraphics[width=\linewidth]{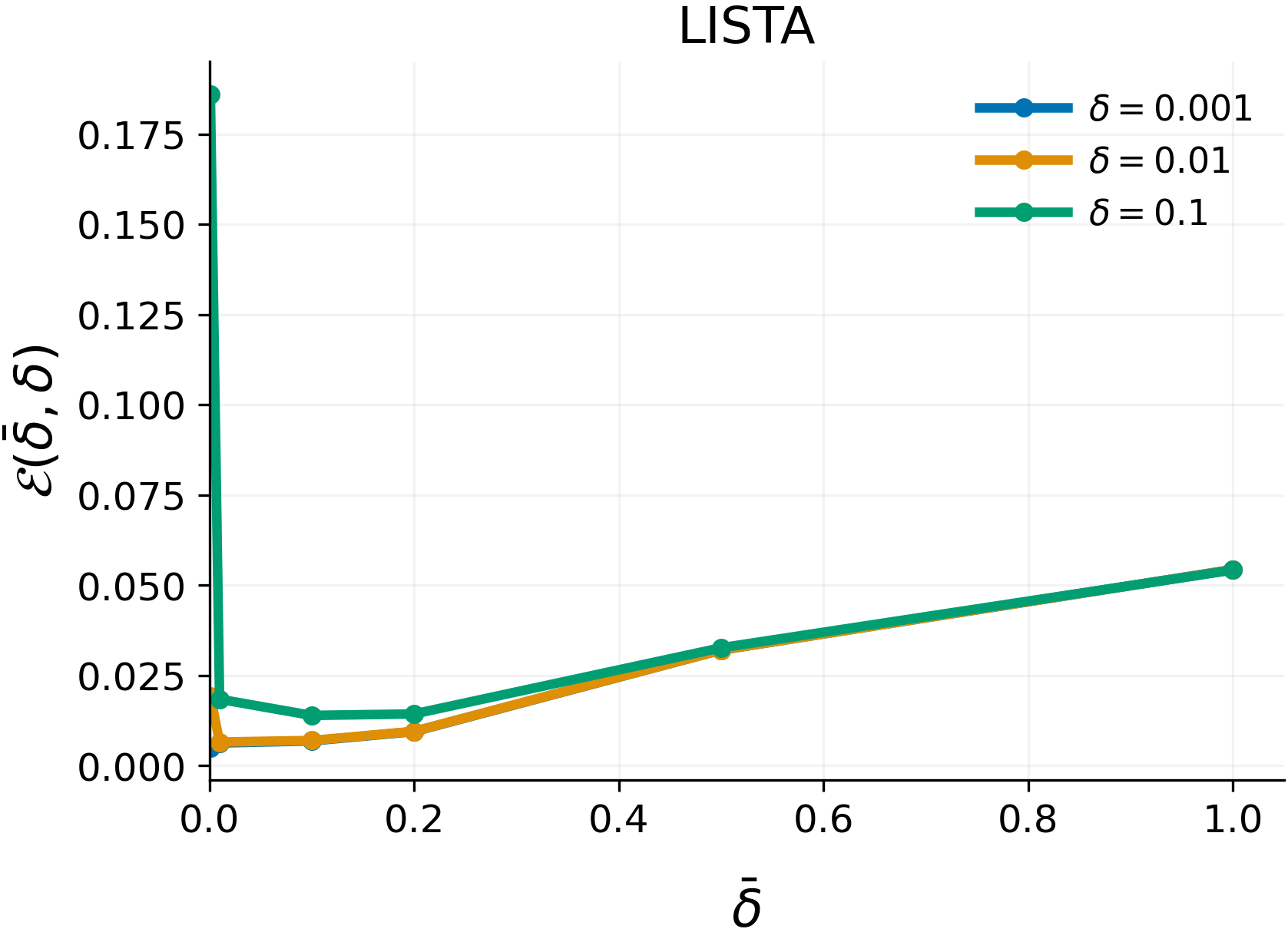}
  \end{subfigure}

  \begin{subfigure}[t]{0.32\textwidth}\centering
    \includegraphics[width=\linewidth]{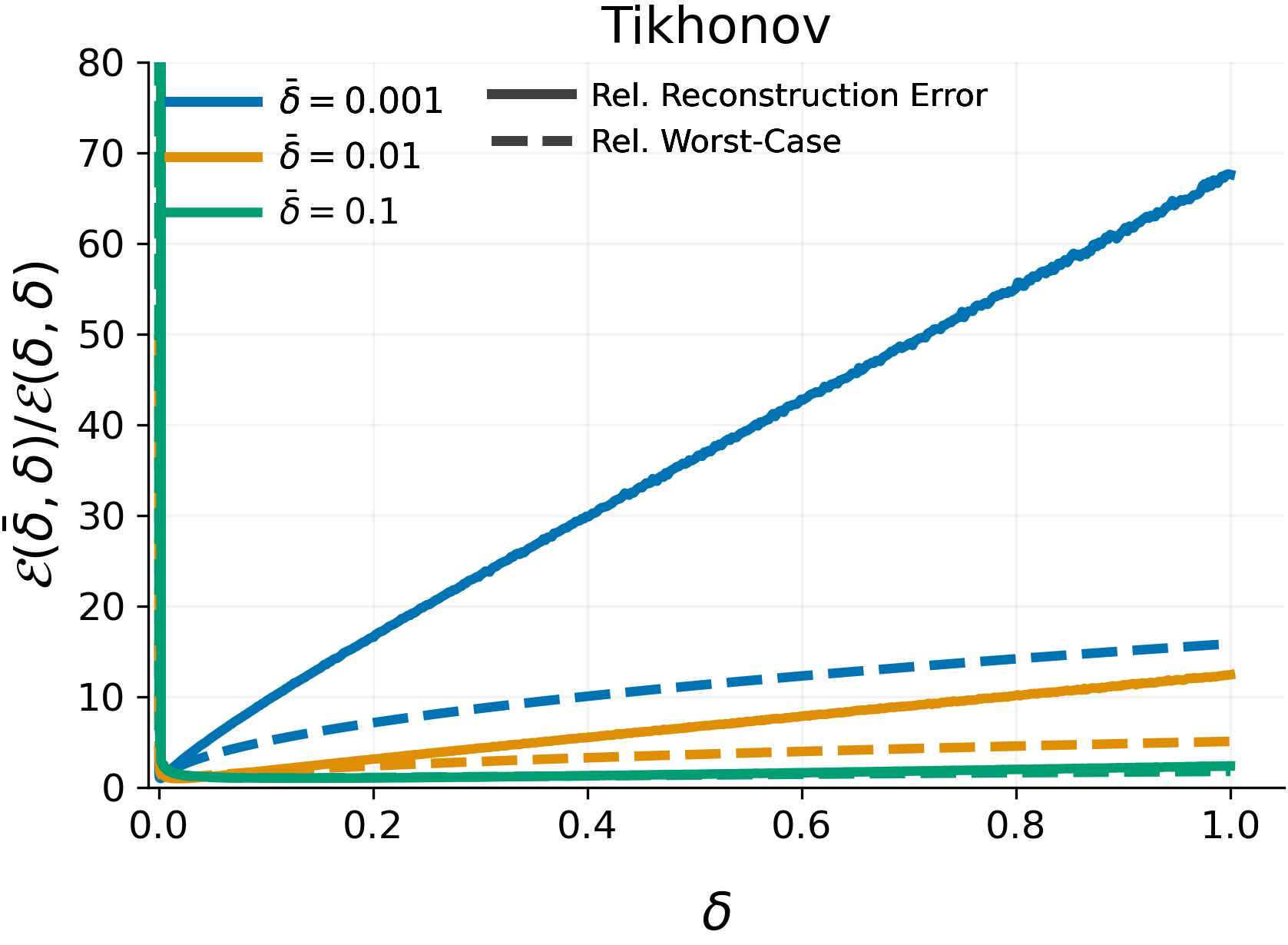}
  \end{subfigure}\hfill
  \begin{subfigure}[t]{0.32\textwidth}\centering
    \includegraphics[width=\linewidth]{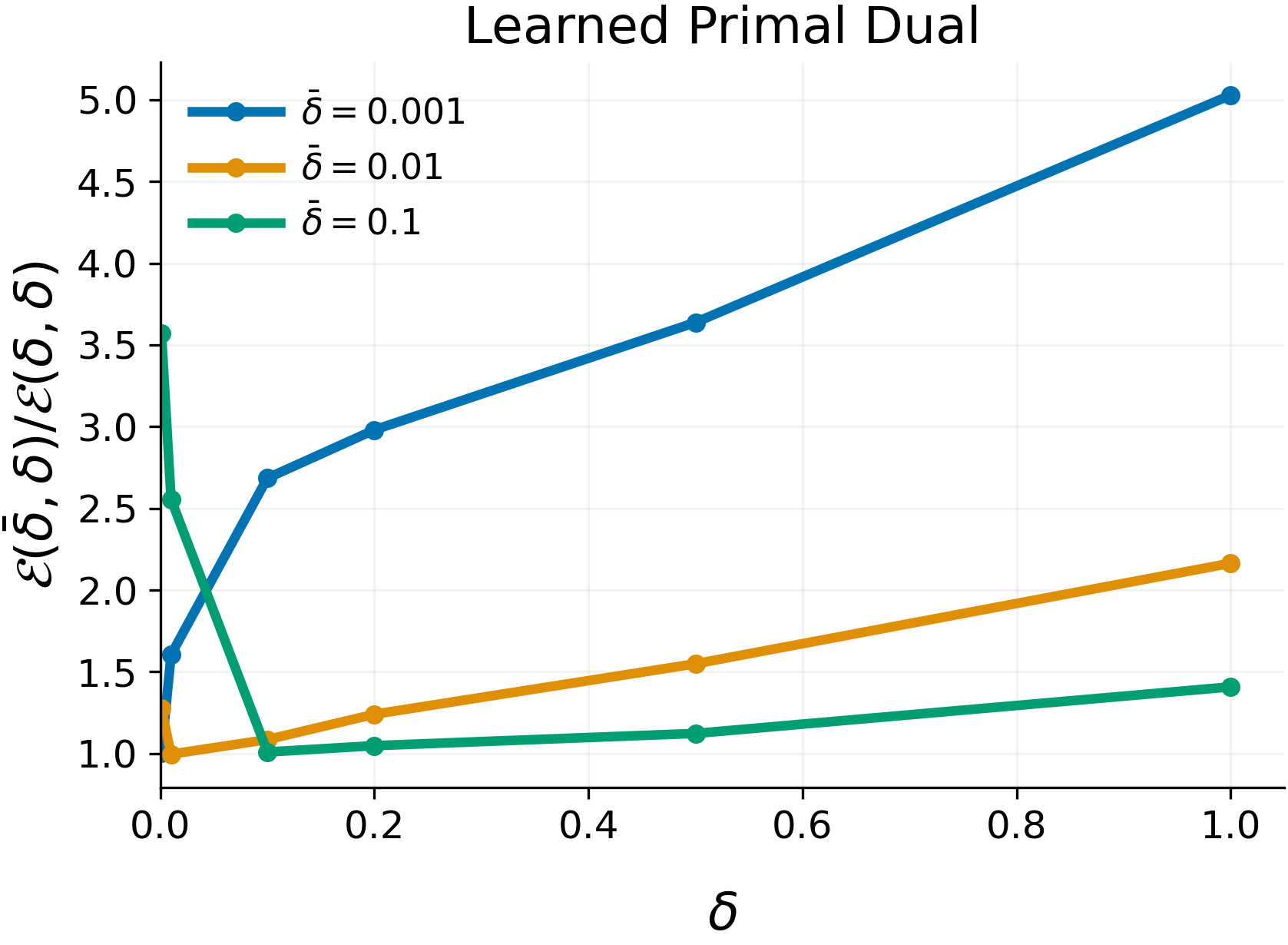}
  \end{subfigure}\hfill
  \begin{subfigure}[t]{0.32\textwidth}\centering
    \includegraphics[width=\linewidth]{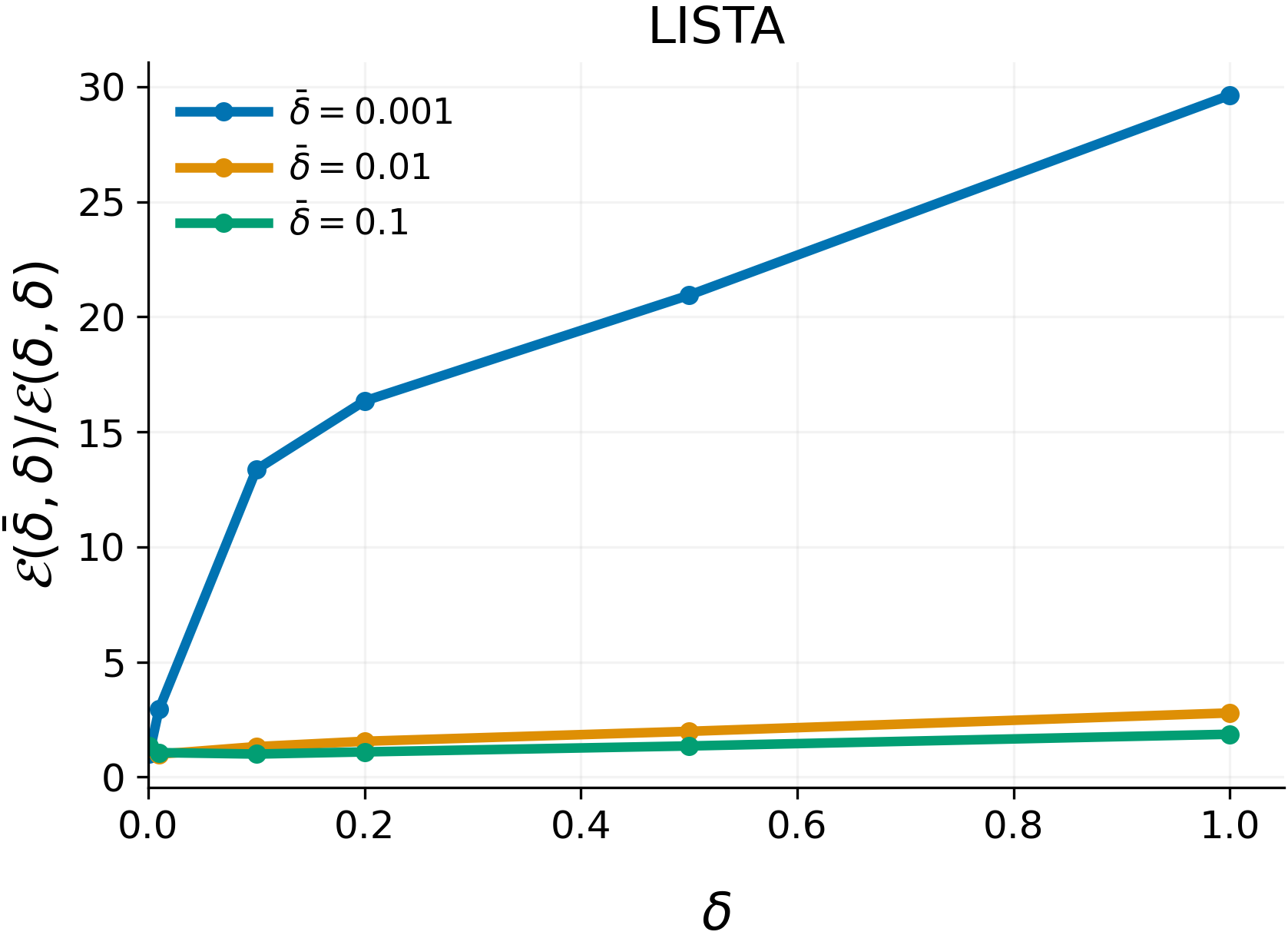}
  \end{subfigure}
  \caption{Reconstruction errors for data in $X_{8}$ and $A_I$ with Tikhonov regularization ($\rho = 0.23$) \emph{(left)}, LPD \emph{(middle)}, and LISTA \emph{(right)}.
  We plot the errors at fixed regularization level $\bar \delta$ over varying $\delta$~\emph{(upper row)} and over varying  $\bar \delta$ at fixed $\delta$~\emph{(middle row)} as well as the corresponding relative errors~\emph{(bottom row)}.
  The dashed lines correspond to the analytical bounds as in Corollary~\ref{cor:wc_error}.}
  \label{fig:wc_grid_int_nd}
\end{figure}

\subsection{Analysis of truncated Tikhonov regularization}\label{subsec:trunc_tikh}
So far, the SVD of $A$ was only used for proving results.
Below, we assume access to it and consider the truncated Tikhonov regularization
\begin{equation}\label{eq:TruncTik}
    T_\alpha^{(M)} y^\delta = \sum_{i=1}^M \frac{\sigma_i}{\sigma_i^2+\alpha}\,\langle y^\delta, u_i\rangle v_i,
\end{equation}
which remains independent of $X_N$.
The truncated Tikhonov regularization \eqref{eq:TruncTik} has two regularization parameters, namely $M$ and $\alpha$.
Setting $\alpha = 0$ yields the classical truncated SVD.
If the singular values of $A$ exhibit a prescribed decay and if $X_N$ is nice (according to Assumption \ref{ass:2}), \eqref{eq:TruncTik} admits explicit error bounds that depend on $M$, $N$ and on the decay rate of the singular values.
The following corollary illustrates this for the case $\sigma_j  = j^{-1}$.

\begin{corollary}\label{cor:ndim_2}
    Assume $\sigma_j = j^{-1}$ for all $j \geq 1$ and that the requirements of Theorem~\ref{thm:n_dim2} hold.
    If $M\geq N$, we obtain the error bound
    \begin{align}
        \| T^{(M)}_{\alpha} y^\delta - x^\dagger \| &\leq \min\Big\{\frac{M\delta}{1+\alpha M^2}, \frac{\delta}{2\sqrt{\alpha}}\Big\} + \min\Big\{N\alpha \rho, \frac{\sqrt{\alpha}}{2}\rho\Big\} \\
        &= \begin{cases}
            \frac{M\delta}{1+\alpha M^2} + N\alpha \rho &\text{if } \sqrt{\alpha} \leq \min \left\{\frac{1}{M}, \frac{1}{2N} \right\}, \notag \\[1.5mm]
            \frac{\delta}{2\sqrt{\alpha}} +  N\alpha \rho  &\text{if } \frac{1}{M} < \sqrt{\alpha} \leq \frac{1}{2N}, \\[1mm] 
            \frac{M\delta}{1+\alpha M^2} + \frac{\sqrt{\alpha}}{2} \rho &\text{if } \frac{1}{2N} < \sqrt{\alpha} \leq \frac{1}{M}, \\[1mm]
            \frac{1}{2}\left( \frac{\delta}{\sqrt{\alpha}} + \sqrt{\alpha} \rho \right) &\text{if } \sqrt{\alpha} > \max\left\{\frac{1}{M}, \frac{1}{2N}\right\}.
        \end{cases}
    \end{align}
\end{corollary}
\begin{proof}
    The estimate of the approximation error follows directly from Theorem~\ref{thm:n_dim2}.
    Thus, it remains to refine the data error estimate.
    The function $F(x)=\frac{x}{1+\alpha x^2}$ is increasing on $(0,1/\sqrt{\alpha}]$ and decreasing on $[1/\sqrt{\alpha},\infty)$.
    Hence, it holds that
    \begin{equation}
        \|T^{(M)}_\alpha(y^\delta-y^\dagger)\| \leq \sup_{j = 1, \dots, M} \biggl | \frac{\sigma_j}{\sigma_j^2+\alpha}\biggr | \delta  = \sup_{j = 1, \dots, M} \left | \frac{j}{1+\alpha j^2}\right | \delta \leq \begin{cases}
            \frac{M \delta}{1+\alpha M^2} &\text{if } \sqrt{\alpha} \leq \frac{1}{M},\\
            \frac{\delta}{2\sqrt{\alpha}} &\text{if } \sqrt{\alpha} > \frac{1}{M},
        \end{cases}
    \end{equation}
    which concludes the proof.
\end{proof}

Below, we consider the special case that $X_N$ is spanned by the first $N$ singular vectors of $A$.
To this end, we use generic noise description, which includes standard Gaussian noise as special case.

\begin{theorem}\label{lemma:expectation}
    Let $x^\dagger = \sum_{i=1}^N c_i v_i$ with $ c_i \coloneqq \langle x^\dagger,v_i\rangle \neq 0$ for $i=1,\ldots,N$, and assume that the noise-perturbed data are given by $y^\delta = Ax^\dagger + \eta$, where $\eta$ defines random noise with probability density $p_H$. Suppose that for $\eta_i \coloneqq \langle \eta,u_i\rangle$ it holds 
    \begin{equation}
        \E_{p_H}[\eta_i] = 0, \qquad \E_{p_H}[\eta_i^2] = \beta_{\eta,i}^2 < \infty \qquad \text{for all } i \in \N.
    \end{equation}
    Then, for $E_M \coloneqq T_\alpha^{(M)}y^\delta-x^\dagger$ with $2\alpha \geq \max\{ 0, \max_{m < N} \beta_{\eta,m+1}^2 / c_{m+1}^2 - \sigma_{m+1}^2\}$, it holds
    \begin{equation}\label{eq:alpha_cond}
        \E_{p_H}\bigl[\|E_N\|^2] \leq \E_{p_H}\bigr[\|E_M \|^2] \qquad \text{for all } M \in \N. 
    \end{equation}
\end{theorem}
\begin{proof}
We have
\begin{equation}
    \langle y^\delta,u_i\rangle =
        \begin{cases}
        \sigma_i c_i + \eta_i, & i \le N,\\
        \eta_i, & i>N.
        \end{cases}
\end{equation}
and 
\begin{equation}
    E_M = T_\alpha^{(M)}y^\delta-x^\dagger  = \sum_{i=1}^M \frac{\sigma_i}{\sigma_i^2+\alpha} \eta_i v_i + \sum_{i=1}^{\min\{M,N\}}\Bigl(-\frac{\alpha}{\sigma_i^2+\alpha}\Bigr)c_i v_i - \sum_{i=M+1}^N c_i v_i.
\end{equation}
Thus, the coefficients of $E_M$ in the orthonormal basis $(v_i)_i$ are
\begin{equation}\label{eq:coeffs}
    e_i^{(M)} \coloneqq
\begin{cases}
-\dfrac{\alpha}{\sigma_i^2+\alpha} c_i
+ \dfrac{\sigma_i}{\sigma_i^2+\alpha}\,\eta_i, & i \le \min\{M,N\},\\[0.7em]
\dfrac{\sigma_i}{\sigma_i^2+\alpha}\,\eta_i, & N < i \le M,\\[0.4em]
-c_i, & M < i \le N,\\[0.2em]
0, & i> \max\{M,N\}.
\end{cases}
\end{equation}
Since $(v_i)_i$ is orthonormal, we have
\begin{equation}
    \|E_M\|^2 = \sum_{i\ge 1} |e_i^{(M)}|^2 \quad \text{ and } \quad \E_{p_H} \bigl[\|E_M\|^2\bigr] = \sum_{i\ge 1} \E_{p_H} \bigl[|e_i^{(M)}|^2\bigr].
\end{equation}
For $i \le M$, we can use \eqref{eq:coeffs} and $\E_{p_H}[\eta_i]=0$ to obtain
\begin{equation}\label{eq:ErrorExpect}
    \E_{p_H}\bigl[|e_i^{(M)}|^2\bigr] = a_i^2 c_i^2 + b_i^2 \beta_{\eta,i}^2, \qquad a_i \coloneqq \frac{\alpha}{\sigma_i^2+\alpha}, \,\, b_i \coloneqq \frac{\sigma_i}{\sigma_i^2+\alpha}.
\end{equation}

\paragraph{Case $M \ge N$.}
Using \eqref{eq:ErrorExpect} and  that $c_i = 0$ for $i > N$, we get
\begin{equation}
\E_{p_H}\bigl[\|E_M\|^2\bigr]
= \sum_{i=1}^N \bigl(a_i^2 c_i^2 + b_i^2 \beta_{\eta,i}^2\bigr)
   + \sum_{i=N+1}^M b_i^2 \beta_{\eta,i}^2
\eqcolon C + \sum_{i=1}^M b_i^2 \beta_{\eta,i}^2,
\end{equation}
where $C \coloneqq \sum_{i=1}^N a_i^2 c_i^2$ is independent of $M$.
Similarly, we obtain
\begin{equation}
    \E_{p_H}\bigl[\|E_{M+1}\|^2\bigr]
= C + \sum_{i=1}^{M+1} b_i^2 \beta_{\eta,i}^2.
\end{equation}
Hence, we get that
\begin{equation}
    \E_{p_H}\bigl[\|E_{M+1}\|^2\bigr] - \E_{p_H}\bigl[\|E_M\|^2\bigr]= b_{M+1}^2 \beta_{\eta,M+1}^2 = \Bigl(\frac{\sigma_{M+1}}{\sigma_{M+1}^2+\alpha}\Bigr)^2 \beta_{\eta,M+1}^2 \ge 0.
\end{equation}
Thus, we have $\E_{p_H}[\|E_{M}\|^2] \geq \E_{p_H}[\|E_N\|^2]$ for all $M \ge N$ and all $\alpha > 0$.

\paragraph{Case $M < N$.}
Using \eqref{eq:ErrorExpect} and  \eqref{eq:coeffs}, we get
\begin{equation}
\E_{p_H}\bigl[\|E_M\|^2\bigr] = \sum_{i=1}^M \bigl(a_i^2 c_i^2 + b_i^2 \beta_{\eta,i}^2\bigr) + \sum_{i=M+1}^N \E_{p_H}\bigl[|e_i^{(M)}|^2\bigr] = \sum_{i=1}^M \bigl(a_i^2 c_i^2 + b_i^2 \beta_{\eta,i}^2\bigr)+ \sum_{i=M+1}^N c_i^2.
\end{equation}
For $M+1$, we similarly obtain
\begin{equation}
    \E_{p_H}\bigl[\|E_{M+1}\|^2\bigr] = \sum_{i=1}^{M+1} \bigl(a_i^2 c_i^2 + b_i^2 \beta_{\eta,i}^2\bigr)+ \sum_{i=M+2}^N c_i^2,
\end{equation}
where the second sum is empty if $M+1 =N$.
Thus, we have
\begin{equation}
   \E_{p_H}\bigl[\|E_{M+1}\|^2\bigr] - \E_{p_H}\bigl[\|E_M\|^2\bigr]= (a_{M+1}^2 - 1)c_{M+1}^2 + b_{M+1}^2 \beta_{\eta,M+1}^2.
\end{equation}
In order for $\E_{p_H}[\| E_{M+1}\|^2] - \E_{p_H}[\| E_M\|^2] \le 0 $ to hold, we require 
\begin{equation}
   (a_{M+1}^2 - 1)c_{M+1}^2 + b_{M+1}^2 \beta_{\eta,M+1}^2 = c_{M+1}^2\left( \left( \frac{\alpha}{\sigma_{M+1}^2 + \alpha}\right)^2 - 1\right) + \left( \frac{\sigma_{M+1}}{\sigma_{M+1}^2+\alpha}\right)^2 \beta_{\eta,M+1}^2 \leq 0.
\end{equation}
The latter is equivalent to 
\begin{equation}
    (\sigma_{M+1}^2+\alpha)^2 - \frac{\sigma_{M+1}^2\beta_{\eta,M+1}^2}{c_{M+1}^2} \geq \alpha^2,
\end{equation}
which holds precisely if
\begin{equation}
    2\alpha \geq \frac{\beta_{\eta,M+1}^2}{c_{M+1}^2}-\sigma_{M+1}^2.
\end{equation}
Therefore, if we choose $\alpha$ sufficiently large, we have $\E_{p_H}[\| E_N \|^2] \leq \E_{p_H}[\| E_M \|^2]$ for all $M < N$. 
\end{proof}

Theorem \ref{lemma:expectation} states that there exists a regularization parameter $\alpha$ (depending on $x^\dagger$) for which truncating exactly at the number of relevant singular vectors $N$ minimizes the expected squared reconstruction error over all truncation levels $M \in \N$.
Note that $\alpha$ satisfying~\eqref{eq:alpha_cond} does not necessarily correspond to an optimal regularization parameter and may therefore lead to suboptimal reconstruction results.
Still, one can estimate $N$ by scanning over $M$ and selecting the minimizer of \smash{$\E_{p_H}(\|T^{(M)}_\alpha y^\delta-x^\dagger\|^2)$}.
By doing so for multiple $x^\dagger$, we can estimate the intrinsic data dimension (in terms of relevant basis vectors) based on the SVD basis.
This compensates for the facts that the optimal subspace dimension may vary between samples, and that the reconstruction error can be influenced by sample-specific features.
Note that the procedure is primarily of theoretical interest.
In particular, the operator dependent SVD basis often does coincide with the actual subspace basis for the given data.
Thus, our experiments should be seen mainly as illustration for how the intrinsic dimension influences the behavior of the reconstruction process.

\begin{remark}
The assumptions of Theorem~\ref{lemma:expectation} are in particular satisfied for Gaussian noise. If $\eta \sim \mathcal N(0,\delta^2 I)$ in $Y$, the noise coefficients $\eta_i$ are i.i.d.\ with distribution $\mathcal N(0,\delta^2)$, and hence
\begin{equation}
    \E[\eta_i] = 0, \qquad \E[\eta_i^2] = \delta^2 < \infty \quad \text{for all } i \in \N.
\end{equation}
More generally, if $\eta \sim \mathcal{N}(0,\Sigma)$ with covariance $\Sigma$, then $\E[\eta_i]=0$ and ${\E[\eta_i^2]=\langle \Sigma u_i, u_i\rangle < \infty}$, so that Theorem~\ref{lemma:expectation} applies whenever these variances are finite.
\end{remark}

\begin{remark}\label{rem:other_basis}
Although Theorem~\ref{lemma:expectation} is proved only for the SVD basis, it seems natural to expect similar behavior for other bases that are somehow well-aligned with the singular vectors of $A$.
The reason is that the proof relies on a mode-wise decomposition of the reconstruction error into approximation and noise contributions, expressed in terms of the coefficients $c_i = \langle x^\dagger, v_i \rangle$ and $\eta_i = \langle \eta, u_i \rangle$.
Thus, if the scalar products $\langle b_i,  v_j\rangle_{ij}$ are sufficiently localized at small $j$, then the transformed coefficients still interact predominantly with large singular components.
This can be seen as kind of an identifiability condition.
In that case, the reconstruction error is expected to exhibit a behavior similar to the SVD setting, even though additional cross-terms appear.
\end{remark}

\subsection{Numerical results for subspace-constrained reconstruction}~\label{subsec:num_results}

Throughout, we consider the discrete Radon operator $A_R$ and the discretized $L^2$-norm introduced in Section~\ref{subsec:num_exp_full}. 
The singular values of the Radon operator are known to decay polynomially, specifically as $\sigma_N \sim n^{-1/2}$ (see, e.g., \cite{Louis1989}), which we therefore also expect approximately for $A_R$.
Thus, when using the SVD basis, Assumption~\ref{ass:2}(ii) is naturally satisfied.
As discussed in Remark~\ref{rem:other_basis}, while the theoretical results are derived for the SVD basis, we expect similar qualitative behavior for other bases that are  well aligned with the singular vectors. 
Without making this formal, we thus extend our experiments to finite-dimensional linear subspaces
\begin{equation}
    X_{N} \coloneqq \mathrm{span}\{b_1,\dots,b_{N}\} \subset X = \R^n,
\end{equation}
where the basis $(b_j)_{j=1}^n$ can be for example the
\begin{itemize}
    \item[(a)] SVD basis $b_j = v_j$ for $j=1,\dots,n$; or
    \item[(b)] the Coordinate basis $b_j = e_{k_j}$, where $(e_k)_{k=1}^n$ denotes the unit vectors and $(k_j)_{j=1}^n$ is a random permutation of $\{1,\dots,n\}$\footnote{The resulting permuted basis remains fixed to ensure consistency of the subspaces $(X_M)_M$}; or
    \item[(c)] the first $n$ principal components~(PCA) for the given data.
\end{itemize}

In the reconstruction stage, we perform Tikhonov regularization for the operator $A_R$ restricted to the subspace $X_M=\mathrm{span}\{b_1,\dots,b_M\}$ with $M \leq n$ by minimizing the restriction of the Tikhonov functional to $X_M$, namely
\begin{equation}\label{eq:ndim_rec}
    w_{i,\alpha}^\delta = \argmin_w \frac{1}{2}\| y_i^\delta - A_Mw \|^2_Y + \frac{\alpha}{2}\|B_M w\|_X^2 \quad \text{ and } \quad  x^\delta_{i, \alpha} = B_M w_{i,\alpha}^\delta,
\end{equation}
where $B_M \coloneqq [b_1, \hdots, b_M]$  and $A_M \coloneqq A B_M$ for $M \leq n$.
If $B_M=[v_1,\dots,v_M]$ as in (a), then the minimization of \eqref{eq:ndim_rec} reduces to the truncated Tikhonov regularization \eqref{eq:TruncTik}, namely
\begin{equation}
    x^{\delta}_{i,\alpha} = T^{(M)}_\alpha y_i^\delta 
    = \sum_{j=1}^M \frac{\sigma_j}{\sigma_j^2+\alpha}\,\langle y^{\delta}_i,u_j\rangle_Y\, v_j.
\end{equation}

For our experiments, we consider two types of data: (i) simulated signals that lie exactly in $X_{N}$, and (ii) the MNIST data set.
Regarding the simulated data, given a basis $(b_j)_{j=1}^n$ (we use (a) and (b) for our experiments), we fix $N \leq n$ and generate samples 
\begin{equation}
    x_i \coloneqq \sum_{j=1}^{N} c^j_i b_j \in X_N \quad \text{with }
     c^j_i \sim \mathcal U[-1,1], \quad \text{for } i = 1,\dots,L.
\end{equation}
Then, we compute noisy measurements as $y_i^\delta = A_Rx_i + \eta_i$ with ${\eta_i \sim \mathcal{N}(0,\delta^2 I_m)}$.
From this, we try to to estimate the dimension based on the same basis $(b_j)_{j=1}^n$.
MNIST images are highly structured and thus effectively low-dimensional.
Thus, we expect that they can be also well-approximated by a subspace of small dimension $N$. 
To test this, we use the theoretically covered SVD basis as in (a), and additionally compare against a more appropriate basis formed by the first $M$ principle components for the MNIST training set.
We average over 100 reconstructions per data sample, using different noise realizations sampled from $\mathcal{N}(0,\delta^2 I_m)$.
This is done for $L=50$ different samples, computing \smash{$\mathcal{E}(\alpha,\delta) = \frac{1}{L}\sum_{i=1}^L\|x_{i,\alpha}^{\delta} - x_i \|_X$}. 
Note that Assumption~\ref{ass:2}(i) is only strictly fulfilled for the simulated data, where the source condition is explicitly enforced. Consequently, among the considered settings, only the combination of simulated data and the SVD basis fully satisfies Assumption~\ref{ass:2} and thus the requirements of Theorem \ref{lemma:expectation}.

\begin{figure}[t]
\begin{subfigure}{\textwidth}
    \centering
    \includegraphics[width=0.85\linewidth]{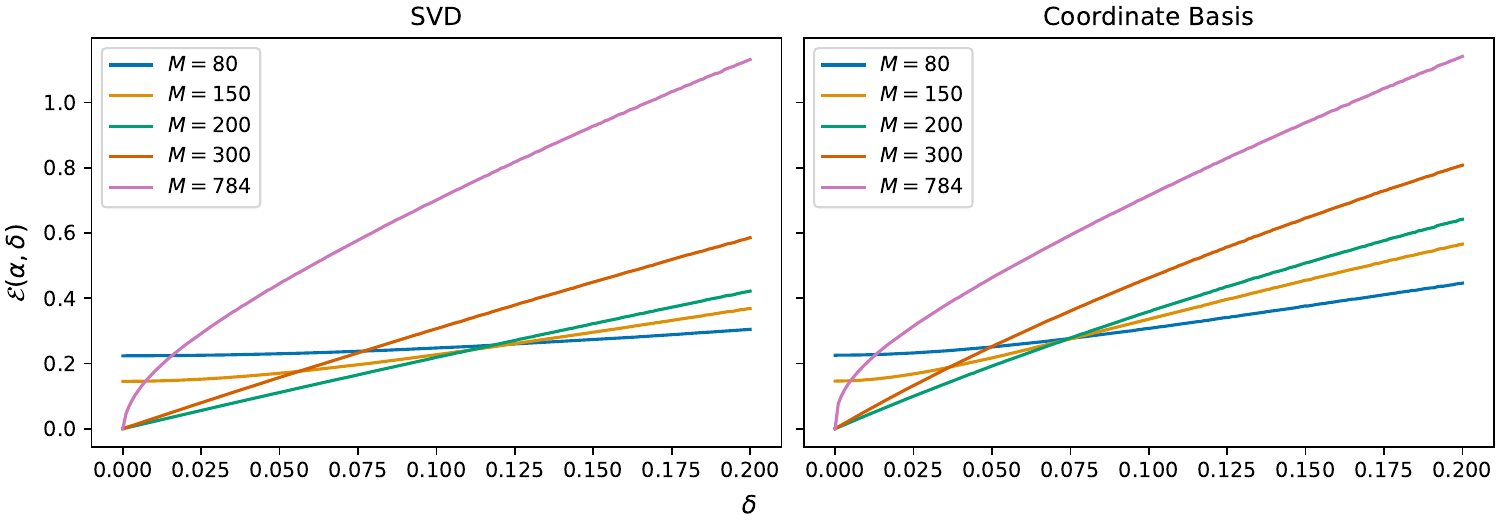}
    \caption{Simulated data with $N=200$  and regularization parameter $\alpha(\delta) = 0.01\delta$. \emph{(left)}~Using the first $M$ components of the SVD basis (a); \emph{(right)}~using $M$ the permuted unit vector basis (b).}
    \label{fig:reddim_simulated}
\end{subfigure}

\begin{subfigure}{\textwidth}
    \centering
    \includegraphics[width=0.85\linewidth]{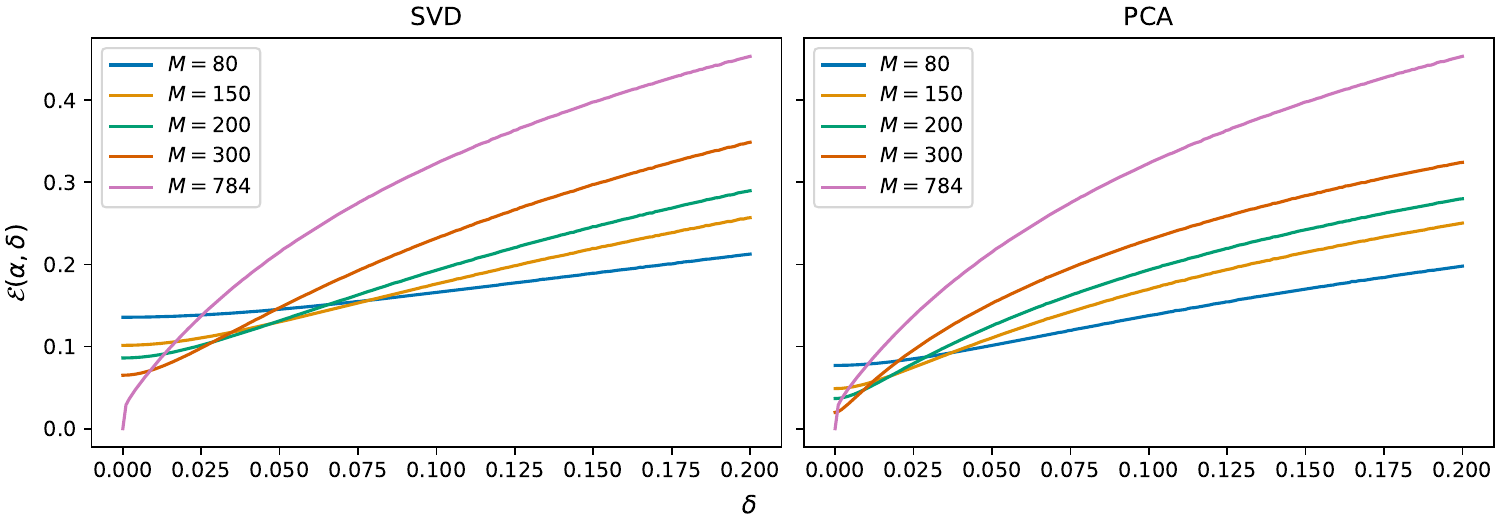}
    \caption{MNIST data set with regularization parameter $\alpha(\delta) = 0.1\delta$. \emph{(left)}~Using the first $M$ components of the SVD basis (a); \emph{(right)}~using the first $M$ principle components for the training set.}
    \label{fig:reddim_mnist}
\end{subfigure}
\caption{Reconstruction errors using truncated Tikhonov regularization for simulated data \emph{(top)} and MNIST (\emph{bottom}) plotted against varying noise levels defined by $\delta$.
In all subplots, we fix one basis for reconstruction and data generation.}
\label{fig:reddim_overdelta}
\end{figure}

The errors $\mathcal{E}(\alpha,\delta)$ for both data sets with respect to each selected basis are shown in Figure~\ref{fig:reddim_overdelta}.
For the simulated data, we observe that reconstructing in the true intrinsic dimension $N=200$ yields consistently smaller errors than using any $M>N$. 
For $M<N$, truncation becomes advantageous once the noise level exceeds a certain threshold, reflecting the classical trade-off between data and approximation error: beyond that point, including additional components primarily amplifies noise rather than improving approximation.
For the MNIST data, we see a similar trend.
Still, in the very low-noise regime, the full ambient dimension performs best.
This is consistent with the fact that MNIST images lie only approximately in a small linear subspace.
Interestingly, when using PCA, the trade-off between subspace dimension and reconstruction error becomes visible already at smaller noise levels.
This can be explained by the fact that PCA is a data-based basis that is well-adapted to the MNIST data compared to the SVD basis, which is instead adapted to the operator $A_R$.
In particular, PCA orders the directions by decreasing data variance, so that the low-variance components mostly fit noise rather than meaningful image structure, making a smaller $M$ preferable.

\begin{figure}[t]
\begin{subfigure}{\textwidth}
    \centering
    \includegraphics[width=0.85\linewidth]{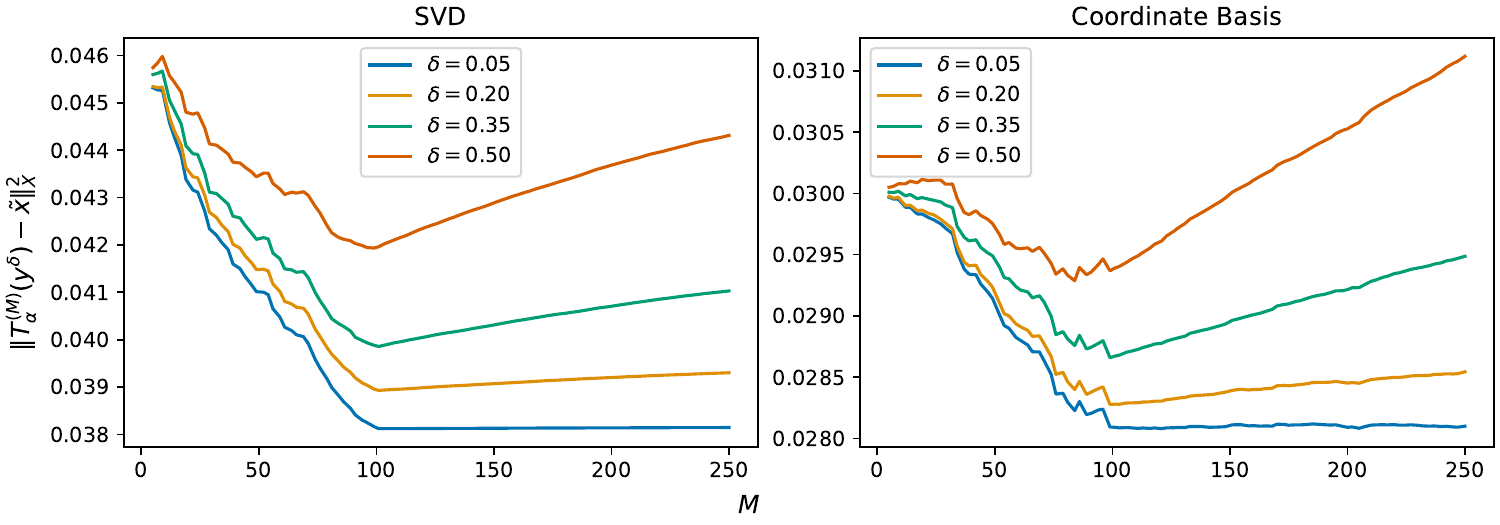}
    \caption{Simulated data sample with $N = 100$.}
    \label{fig:lemma_sim}
\end{subfigure}

\vspace{5mm}
\begin{subfigure}{\textwidth}
    \centering
    \includegraphics[width=0.85\linewidth]{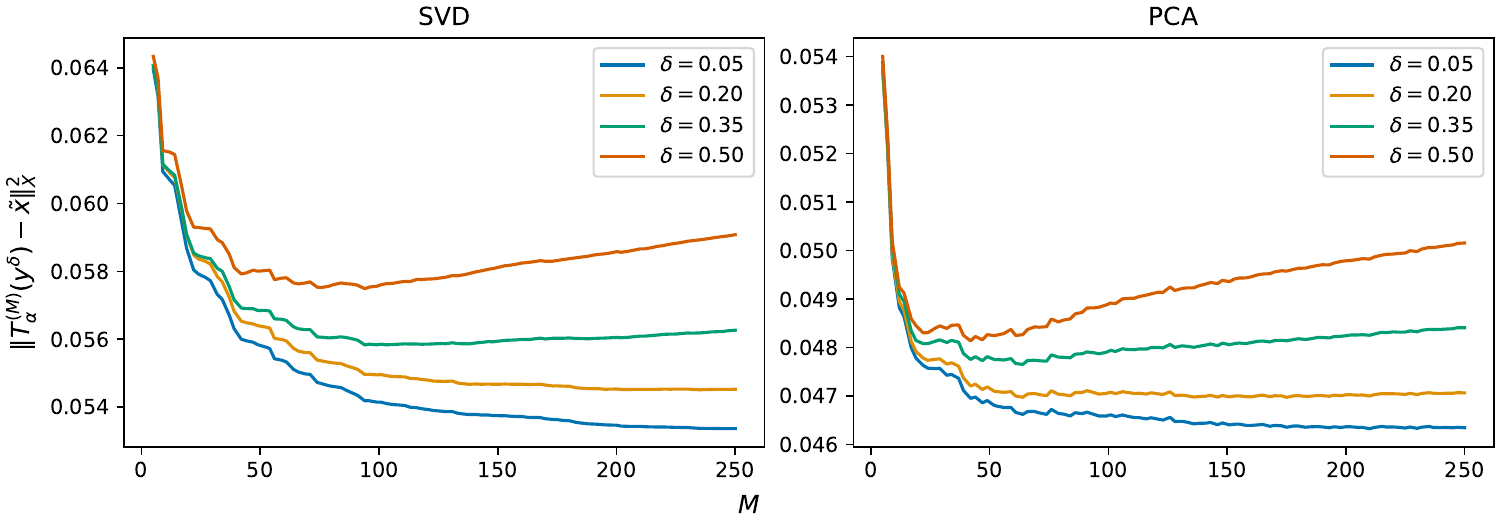}
    \caption{MNIST test sample.}
    \label{fig:lemma_MNIST}
\end{subfigure}
\caption{Reconstruction errors of one sample measured against a reference reconstruction $\tilde x = T_{\alpha_\mathrm{ref}} y^{\delta_\mathrm{ref}}$ with $\alpha_\mathrm{ref} = 0.03$, $\delta_\mathrm{ref} = 0.01$ and regularization parameter $\alpha = 0.5$ for all experiments.
We show the reconstruction in~\eqref{eq:ndim_rec} from a measurement $y^\delta$ against $M$, applied to selected noise levels $\delta$.
Although only satisfied for the SVD basis, by abuse of notation, we denote the reconstruction in~\eqref{eq:ndim_rec} by \smash{$T_\alpha^{(M)} (y^\delta)$} with respect to all basis choices.} 
\label{fig:lemma}
\end{figure}

As an extension, we now aim to estimate the intrinsic data dimension $N$ in the spirit of Theorem~\ref{lemma:expectation} as discussed above.
We emphasize again that these experiments are primarily intended as a proof of concept, illustrating how the theoretical findings manifest in practice.
In particular, if a data-adapted basis (such as the one obtained via PCA) is already available, dimension estimation should be performed directly in that basis.
On the other hand, if no ground truth reconstruction is known, our approach offers a way of estimating the intrinsic dimension directly from the data.
For our experiment, we draw one sample (i.e., $L=1$) and average reconstruction errors over $100$ independent noise realizations with a potentially non-optimal reconstruction parameter $\alpha$.
Instead of comparing the reconstructions from~\eqref{eq:ndim_rec} directly to $x^\dagger$, we use as reference an approximate solution
\begin{equation}
    \tilde x \coloneqq T_{\alpha_\mathrm{ref}} y^{\delta_{\mathrm{ref}}}
\end{equation}
computed at a small reference noise level $\delta_{\mathrm{ref}}$.
In contrast to Theorem~\ref{lemma:expectation}, this mimics the practical situation where the true solution is unknown and must be approximated by a high-quality reference reconstruction.

Figure~\ref{fig:lemma} shows the resulting errors for a fixed, sufficiently large regularization parameter $\alpha$.
For the simulated data, we observe a clear minimum at $M = N$ across all tested noise levels.
This confirms the prediction of Theorem~\ref{lemma:expectation} that the expected reconstruction error becomes minimal when the reconstruction dimension matches the intrinsic dimension of the data.
Note that Theorem~\ref{lemma:expectation} guarantees the existence of a suitable $\alpha$ only for basis~(a).
Nevertheless, the same behavior is observed empirically for basis~(b).
For MNIST, a comparable kink becomes apparent only at sufficiently large noise levels, again reflecting the approximation-data-error trade-off and indicating that MNIST does not lie exactly in a low-dimensional subspace. 
Consequently, the transition around an effective dimension is less pronounced and depends on the noise regime.

\section{Learned regularization terms}
As mentioned in the introduction,
 our work is motivated by a paper on learned regularization schemes~\cite{Goujon2022} and its extension \cite{goujon2024learning}.
Below, we perform the associated worst-case error analysis in a slightly simplified setting.
More precisely, we consider the Tikhonov functional (also known as the generalized LASSO \cite{AliTib2019})
\begin{equation} \label{learnedTik}
J_\alpha (x; y^\delta) \coloneqq \| Ax - y^\delta \|^2 + \alpha(\delta)  \Vert Wx \Vert_1,
\end{equation}
where the learnable potentials proposed in  \cite{Goujon2022} are replaced by the absolute value.
Interestingly, their differentiable learned potentials consistently resemble smoothed versions of the absolute value, commonly referred to as the Huber potential.
This retrospectively justifies the simplification.
However, this makes the functional \eqref{learnedTik} non smooth.
Still, it can be efficiently minimized using primal-dual methods \cite{Condat2013}.
Now, we detail the learning and reconstruction pipeline.

First, the sparsifying transform $W$ is learned for a denoising task with fixed noise level $\delta$, i.e., no knowledge of $A$ is used in the training (so-called \emph{universal training}).
Due to the imaging context, the authors of \cite{Goujon2022} restrict $W$ to be in the subspace of convolution operators (with multiple channels).
The training is performed with pairs \smash{$(x_i^{\bar \delta}, x_i)$}, $i=1,\ldots,N$, of noisy and corresponding ground truth images, which are generated synthetically.
The noise level $\bar \delta$ is the same for all tuples (in principle, training on multiple noise levels is also possible).
As training problem, we consider 
\begin{equation}
    \hat W  = \argmin_W  \sum_{i=1}^N \|\hat x_i - x_i\|^2  \qquad \text{s.t.}\qquad\hat x_i = \argmin_x \| Ax - x_i^{\bar \delta} \|^2 +  \Vert Wx \Vert_1.
\end{equation}
Solving this bilevel problem requires implicit differentiation techniques or unrolling, see for example \cite{LearnedRegularizers} for details.

To solve an inverse problem, a grid search is used to determine suitable values  for $\alpha(\delta)$ based on the knowledge of $A$ and $\delta$ (\emph{fine-tuning}).
This requires only a few samples $(x,y^\delta)$ with $y^\delta=Ax + \eta$ per noise level.
This phase is similar to other concepts for choosing optimal regularization parameters.
If we even construct a function $\alpha \colon \R \to \R$ (for example via linear spline interpolation), this leads to an a-priori parameter choice rule.
The model is then deployed to new data $y^\delta$ as follows.
First, the knowledge of the noise level $\delta$ is used to determine $\alpha(\delta)$.
Then, a reconstruction $x_\alpha^\delta$ is computed as a minimizer of the functional $J_\alpha$.

Now, we analyze the resulting regularization scheme 
\begin{equation}\label{eq:RecProblem}
    x_\alpha^\delta = \argmin_x J_\alpha (x; y^\delta) =  \argmin_x \| Ax - y^\delta \|^2 + \alpha(\delta)  \Vert Wx \Vert_1.
\end{equation}
Note that any component of $x$ in the orthogonal complement of the span $X_W= \mathrm{span} \{ w_j\}$ of the rows $w_j$ of $W$ does not influence the trained regularizer.
Below, we consider the cases $X_W = X$ and $ X_W \neq X$ separately.
If $X_W=X$, then $\Vert Wx\Vert_1 $ is indeed a norm.
A proper choice of $\alpha$ then yields an optimal regularization scheme under general conditions, see e.g.~\cite{Scherzer2009}.
This explains that learning  $W$ on a denoising problem and then applying it -- with an adapted parameter choice $\alpha$ -- to general inverse problems is a feasible approach.
If $X_W \neq X$, however, we expect an unbounded worst-case error as exemplified in the following theorem. 

\begin{theorem}\label{thm:UnboundedErrorCRR}
    Assume that $\mathrm{range}(AX_W^\perp)$ is non-closed in $Y$ and let $y^\delta \in \overline{A X_W^\perp}\setminus AX_W^\perp$. Then, there exists a sequence $(x_n)_{n\in\N} \subset X_W^\perp$ such that $\|Ax_n - y^\delta \| \xrightarrow{n\to\infty} 0$ and $J_\alpha(x_n; y^\delta) \xrightarrow{n\to\infty} 0$, but also $\|x_n\| \xrightarrow{n\to \infty}\infty$ and $\|x_n - x^\dagger \| \xrightarrow{n\to \infty}\infty$.
\end{theorem}

\begin{proof}
    The proof follows from standard arguments, see e.g.~\cite{Louis1989} or~\cite{EnglHankeNeubauer1996}.
    We include it for convenience.
    Since \smash{$y^\delta \in \overline{A X_W^\perp}$}, there exists a sequence $(x_n)_{n\in\N} \subset X_W^\perp$ with $\lim_{n \to \infty}\|Ax_n - y^\delta\| = 0 $.
    We show that for any such sequence it holds $\lim_{n \to \infty}\|x_n\| = \infty$. 
    Suppose for contradiction that $(\|x_n\|)_{n\in\N}$ is bounded.
    Then there exists a subsequence $(x_{n_k})_{k\in\N}$ and $x \in X_W^\perp$ with $x_{n_k} \rightharpoonup x$.
    Since $A$ is linear and bounded, it is weak-to-weak continuous and thus $Ax_n \rightharpoonup Ax$.
    By uniqueness of the strong limit, we must have $Ax = y^\delta$. Since $x \in X^\perp_W$, we have $y^\delta \in A X_W^\perp$.
    This contradicts $y^\delta \notin A X_W^\perp$.
    Therefore, $\lim_{n \to \infty}\|x_n\| = \infty$ and also ${\|x_n - x^\dagger\| \geq |\|x_n \| - \|x^\dagger \|| \rightarrow \infty}$. 
\end{proof}
Theorem \ref{thm:UnboundedErrorCRR} refers to the idealized infinite dimensional case.
It implies that we should choose a learnable transform $W$ that adapts to the input size.
Indeed, the convolution has this feature.
To the best of our knowledge, there do not exist satisfying stability results for the infinite dimensional case if the sparsifying transform $W$ is not injective.

For the finite dimensional case, the literature is much richer.
It holds that given data $y^\delta$, all solutions $\hat x$ lead to the same value of $A \hat x$ and $\Vert W \hat x \Vert_1$, see \cite[Lem.\ 1]{AliTib2019}.
Since for any $\hat x$ the $A \hat x$ are equal, the optimality relation \begin{equation}\label{eq:Optimality condition}
    0 \in 2A^T(A x - y^\delta) + \alpha(\delta) W^T \partial \Vert \cdot \Vert_1(W \hat x)
\end{equation}
for \eqref{learnedTik} implies that any optimal subgradient $\hat \gamma \in \partial \Vert \cdot \Vert_1(W \hat x)$ leads to the same value of $W^T \hat \gamma$.
A discussion on sufficient conditions for uniqueness of the solution $\hat x$ can be also found in \cite{AliTib2019}.
Below, we do not enforce such conditions.
Following \cite{NeuAlt2024, HuYaoZha2024}, we instead perform a set-valued stability analysis of the reconstruction map $y \mapsto \hat x(y)$.
Our simple and direct proof of Lipschitz continuity relies on novel arguments rooted in convex analysis \cite{BauschkeCombettes2017}.
\begin{proposition}
    The set-valued solution map $y \mapsto \hat x(y)$ is Hausdorff Lipschitz continuous, namely there exists $\kappa > 0$ such that $\hat x(\tilde y) \subset \hat x(y) + \kappa \Vert y - \tilde y \Vert_2 B_1(0)$ for all $y, \tilde y \in \R^m$.
\end{proposition}
\begin{proof}
For any optimal solution $\hat x(y)$ to \eqref{eq:RecProblem} with fixed $y$, we have that
\begin{equation}
    \hat u (y) = A \hat x(y) = \argmin_u \Vert u - y\Vert_2^2 + \alpha \min_x \{\Vert W x \Vert_1 : Ax = u\},
\end{equation}
where as usual the minimum is infinity if no such $x$ exists.
Now, we need to show that the infimal projection $v(u) = \min_x \{\Vert W x \Vert_1 : Ax = u\}$ is proper, convex and lower semi-continuous.
Clearly, we have $v(u)\geq 0$ and $v(0) = 0$, thus $v$ is proper.
The convexity follows from standard arguments \cite[Prop.\ 12.36]{BauschkeCombettes2017}.
Regarding lower semi-continuity, we can equivalently show that the epigraph $\mathrm{epi}(v) = \{(u,t): v(u) \leq t\}$ is closed.
For this, we note that
\begin{equation}
     C = \Bigl\{(u,t,x,z): \sum_i z_i \leq t, -z_i \leq W_{i,:} x \leq z_i, z_i \geq 0, Ax = u\Bigr\}
\end{equation}
is a closed polyhedron.
Due to this structure, the same holds $\mathrm{epi}(v) = P_{1,2} C$, where $P_{1,2}$ denotes the projection onto the first two coordinates of $C$.
Therefore, $y \mapsto \hat u (y) = \mathrm{prox}_{\alpha v}(y)$ is a proximal operator and thus 1-Lipschitz continuous.
To conclude, we note that the solution set $\hat x (y)$ is given by
\begin{equation}
   \hat x (y) = \argmin_x \{\Vert Wx \Vert_1 : Ax = \hat u (y)\},
\end{equation}
which can be rewritten as a linear program (which is by construction feasible).
Now, the claim follows by the Lipschitz stability of the solution set of linear programs with respect to perturbations of the right hand side in the linear equality constraints \cite[Thm.\ 2.4]{ManShi1987}.
\end{proof}
Now, we look into estimates for $\kappa$.
First, we note that $\partial \Vert \cdot \Vert_1$ is either a singleton or a convex polytope.
By performing a decomposition of the domain according to the sign pattern of $Wx$, we obtain an affine condition in $x$ and $y$ from \eqref{eq:Optimality condition}.
At first glance, it appears that the Lipschitz constant on each affine region is bounded independently by $1/\sigma_{\min}(A)$.
However, given a region with associated index set $I$ such that $W_{I,:} x = 0$ for all x in this region, we can obtain the (potentially) improved bound $1/\sigma_{\min}(A P_{\ker(W_{I,:})})$.
This means that fewer non zero entries potentially lead to better stability.
Thus, we need to analyze the size of $I$ depending on $\alpha$ and $y$.
In view of \eqref{eq:Optimality condition}, we get that 
\begin{equation}
    \bigl\Vert W^T \partial \Vert \cdot \Vert_1(W \hat x) \bigr\Vert_2 \leq \frac{2}{\alpha} \Vert A \Vert_2 \Vert y\Vert_2.
\end{equation}
If $W=I$ (namely for LASSO reconstruction) or if $W$ is invertible, this directly leads to an estimate on $\vert I^c \vert$ that scales as $1 / \alpha$.
However, the situation is more challenging in our general setting, and we are not aware of a generic bound.
Even if we would have such a bound, the next step would require stability properties for the restricted operators $\sigma_{\min}(A P_{\ker(W_{I,:})})$ depeding on $\vert I \vert$ (such as the restricted isometry property and robust nullspace property from compressed sensing), which are often unavailable for typical inverse problems.
Thus, we are in principle stuck with the naive upper bound $1/\sigma_{\min}(A)$, namely that we do not have a structural stabilizing effect.
Instead, the variational regularization \eqref{eq:RecProblem} favors certain solutions with desirable properties which empirically greatly enhances reconstruction quality and robustness.
\begin{remark}
    The decomposition techniques proposed in \cite{HuYaoZha2024} enables us to also establish Lipschitz dependence on $\alpha$.
    This only requires a slight modification of the arguments presented above.
\end{remark}

\section{Conclusion}
This paper was motivated by a discussion at the ICSI Conference 2024 regarding the similarities and differences between classical and learned regularization schemes for inverse problems.
This led us to analyze the reconstruction error of Tikhonov regularization with a fixed regularization parameter when applied to data with varying noise levels.
Our proofs closely followed techniques introduced and widely used by Alfred K.\ Louis.
Interestingly, the ill-posedness of the underlying problem has only a minor impact on the reconstruction error.
In particular, for a fixed regularization parameter, the error increases only mildly as the noise level grows.
While this may initially seem surprising, it is in fact consistent with classical theoretical expectations.
We confirmed these results numerically for both Tikhonov regularization and two learned reconstruction methods.

We then turned to the case where the true solutions belong to an unknown but finite-dimensional subspace.
The subsequent error analysis reveals a distinct behavior that depends analytically on the interplay between the noise level, the regularization parameter, and the subspace dimension $N$. In this setting, learned methods do have advantage due to the additional information used in training. This is confirmed by our numerical experiment.
As a side result, this analysis allows the precise determination of $N$ through numerical experiments.
We do not claim that this is the method of choice for determining $N$ -- in fact, we would propose other and more direct techniques for this task.

Finally, we return to the paper that sparked the mentioned discussion at ICSI and that was the starting point for our investigation.
This method starts by learning a regularization term with a sparsifying transform for denoising.
The latter is then fixed and a simple a-priori choice of regularization parameter allows to apply the resulting scheme to general inverse problems. 
Analyzing this approach in the classical infinite-dimensional functional analytic setting shows that no general error bound can be guaranteed.
However, a refined analysis in the discretized setting demonstrates stability, highlighting the practical applicability and potential of this intriguing method.

\section{Acknowledgements}

P.\ Maass acknowledges support from the PIONEER project, funded by the BMFTR (Federal Ministry of Research, Technology and Space) within the VIP+ program (Project No. 03VP13241). S.\ Neumayer acknowledges support from the DFG (grant SPP2298 - 543939932). Finally, the support of the Mobility Programme (M-0187) of the Sino-German Center for Research Promotion is acknowledged.

\printbibliography
\end{document}